\documentclass{amsart}
\usepackage{lipsum}
\makeatletter
\g@addto@macro{\endabstract}{\@setabstract}
\newcommand{\authorfootnotes}{\renewcommand\thefootnote{\@fnsymbol\c@footnote}}%
\makeatother

\usepackage{enumitem}
\usepackage[dvipsnames]{xcolor}
\usepackage{setspace}
\usepackage[english]{babel}
\usepackage{amsmath,amsfonts,amssymb,amsthm,epsfig,epstopdf,titling,url,array}
\usepackage[margin=1.0in]{geometry}
\usepackage{tikz-cd}
\usepackage{calc}
\usepackage{color,hyperref}
\usepackage[noabbrev]{cleveref}
\hypersetup{colorlinks}
\hypersetup{colorlinks={true},linkcolor={red},citecolor=green}
\usepackage{graphicx}
\usepackage{mathrsfs}

\graphicspath{{images/}}
\theoremstyle{plain}
\newtheorem{theorem}{Theorem}

\newtheorem{thm}{Theorem}[subsection]
\newtheorem{lem}[thm]{Lemma}
\newtheorem{prop}[thm]{Proposition}
\newtheorem{cor}[thm]{Corollary}

\newcommand\scalemath[2]{\scalebox{#1}{\mbox{\ensuremath{\displaystyle #2}}}}
\theoremstyle{plain}
\newtheorem{defn}[thm]{Definition}

\theoremstyle{remark}
\newtheorem{rem}[thm]{Remark}

\newcommand{\mycomment}[1]{}
\usepackage{tikz}

\makeatletter
\renewcommand{\tocsection}[3]{%
  \indentlabel{\@ifnotempty{#2}{\bfseries\ignorespaces#1 #2\quad}}\bfseries#3}
\renewcommand{\tocsubsection}[3]{%
  \indentlabel{\@ifnotempty{#2}{\ignorespaces#1 #2\quad}}#3}

\renewcommand{\tocsubsubsection}[3]{%
  \indentlabel{\@ifnotempty{#2}{\ignorespaces#1 #2\quad}}#3}

\newcommand\@dotsep{4.5}
\def\@tocline#1#2#3#4#5#6#7{\relax
  \ifnum #1>\c@tocdepth 
  \else
    \par \addpenalty\@secpenalty\addvspace{#2}%
    \begingroup \hyphenpenalty\@M
    \@ifempty{#4}{%
      \@tempdima\csname r@tocindent\number#1\endcsname\relax
    }{%
      \@tempdima#4\relax
    }%
    \parindent\z@ \leftskip#3\relax \advance\leftskip\@tempdima\relax
    \rightskip\@pnumwidth plus1em \parfillskip-\@pnumwidth
    #5\leavevmode\hskip-\@tempdima{#6}\nobreak
    \leaders\hbox{$\m@th\mkern \@dotsep mu\hbox{.}\mkern \@dotsep mu$}\hfill
    \nobreak
    \hbox to\@pnumwidth{\@tocpagenum{\ifnum#1=1\bfseries\fi#7}}\par
    \nobreak
    \endgroup
  \fi}
\AtBeginDocument{%
\expandafter\renewcommand\csname r@tocindent0\endcsname{0pt}
}
\def\l@subsection{\@tocline{2}{0pt}{2.5pc}{5pc}{}}
\def\l@subsubsection{\@tocline{2}{0pt}{4pc}{5pc}{}}
\makeatother

\usepackage[new]{old-arrows}
\usepackage{bm} 
\usepackage{multicol}
\usetikzlibrary{decorations.pathmorphing}
\DeclareMathOperator{\GL}{GL}

\DeclareMathOperator{\ch}{ch}

\usepackage{mathtools}

\DeclareMathOperator{\vol}{vol}

\DeclareMathOperator{\GSp}{GSp}
\newcommand\blfootnote[1]{%
  \begingroup
  \renewcommand\thefootnote{}\footnote{#1}%
  \addtocounter{footnote}{-1}%
  \endgroup
}

\date{} 

\begin{document}
\hypersetup{citecolor=blue}
\hypersetup{linkcolor=red}

\begin{center}
  \normalsize
 EULER SYSTEMS FOR $\mathrm{GSp}_4$ OVER IMAGINARY 
 QUADRATIC FIELDS
  \normalsize
  \\
  \bigskip
  Alexandros Groutides \par 
  
 Mathematics Institute, University of Warwick\par
\end{center}
\begin{abstract}
     We construct an Euler system attached to general-type cohomological cuspidal automorphic representations of $\GSp_4/\mathbf{Q}$ twisted by a Groessencharacter of an imaginary quadratic field. We then use this to bound strict Selmer groups under standard hypotheses. In addition, our approach gives a way of extending the $\GSp_4\times\GL_2/\mathbf{Q}$ Euler system of Hsu-Jin-Sakamoto to a motivic statement which also covers certain small weights omitted in \textit{op.cit.}. 
\end{abstract}
\blfootnote{We gratefully acknowledge support from the following research grant: ERC Grant No. 101001051—Shimura varieties and the Birch–Swinnerton-Dyer conjecture.}
\vspace{-2em}
\section{Introduction}
The theory of Euler systems is one of the most effective methods for studying the arithmetic of global Galois representations. However, constructing Euler systems is a difficult problem. Despite this, there has been significant progress within the last decade, including but not limited to \cite{Lei_2014}, \cite{Lei_Loeffler_Zerbes_2015}, \cite{hsu2020eulersystemsmathrmgsp4times}, \cite{Loeffler_2021}. In this paper, we construct a new example of an Euler system. More specifically, we construct an Euler system for the four-dimensional Galois representation associated to a general-type cohomological cuspidal automorphic representation of $\GSp_4/\mathbf{Q}$, twisted by a Groessencharacter of an imaginary quadratic field $K$. 

This is motivated by work of Lei-Loeffler-Zerbes. More specifically, in \cite{Lei_Loeffler_Zerbes_2015} the authors modify the classes attached to the Rankin-Selberg convolution of two modular forms (which were previously constructed in \cite{Lei_2014}) to obtain an Euler system for a single modular form twisted by a Groessencharacter of an imaginary quadratic field. In this paper we identify and apply an analogous modification to the $\GSp_4\times \GL_2$ classes of Hsu-Jin-Sakamoto \cite{hsu2020eulersystemsmathrmgsp4times}, which we then use to construct an Euler system for $\GSp_4$ over an imaginary quadratic field. Naturally, due to the more complicated nature of the group, there are several obstacles that arise along the way. Thus, our techniques differ in several places from the ones in \cite{Lei_Loeffler_Zerbes_2015}. Key ingredients to our approach include work of Gejima \cite[\S $3$]{Gejima2018AnEF} on unramified Shintani functions for $(\GSp_4, \GL_2\times_{\GL_1}\GL_2)$, and work of Loeffler-Pilloni-Skinner-Zerbes \cite[\S $8$]{loeffler2021higher} which relates Novodvorsky's integral representation of the degree eight local $L$-factor for $\GSp_4\times \GL_2$, to Piatetski-Shapiro's integral representation of the degree four local spin $L$-factor for $\GSp_4.$

Assuming non-vanishing of our Euler system, we obtain bounds on strict Selmer groups as an immediate application. For simplicity, we give these bounds under the additional asusmption that $\Pi$ has level one. This is purely to keep the bounding argument short, and the conditions on the weight of $\Pi$ as simple to state as possible. The more general case will be covered in future work on bounding Bloch-Kato Selmer groups. For clarification, we note that the level one assumption is \textit{not} imposed for the construction of our Euler system which is the main input for this application.
\subsection{Outline of the construction and main results} Let $K$ be an imaginary quadratic field of discriminant $\Delta_K$ and $\psi$ be a Groessencharacter of $K$ of infinity type $(-1,0)$ and modulus $\mathfrak{f}$. 
\begin{rem}
    We should note that our techniques allow for infinity types $(-1-d,0)$ with $d\in\mathbf{Z}_{\geq 0}$, but in this paper to ease notation and exposition we stick to the case $d=0$ as in \cite{Lei_Loeffler_Zerbes_2015}. We refer the reader to \cite[\S $3.1$]{marannino2025diagonalcyclesanticyclotomictwists} for an idea of how to extend the CM-patching of \cite[\S $5.2$]{Lei_Loeffler_Zerbes_2015} to $d>0$.
\end{rem}
We let $g_\psi$ be the associated cuspidal eigenform, of weight $2$ and level $\Gamma_1(N_{K/\mathbf{Q}}(\mathfrak{f})\Delta_K).$ Let $\Pi$ be a general-type cuspidal automorphic representation of $\GSp_4/\mathbf{Q}$ with $\Pi_\infty$ cohomological unitary discrete-series representation of weight $(k_1,k_2)$ with $k_1\geq k_2\geq 3$. Let $S$ be a finite set of primes such that $\Pi$ is unramified away from $S$. Set $N_\Pi:=\prod_{\ell\in S}\ell$, and let $U^{\GSp_4}(N_\Pi)\subseteq \GSp_4(\hat{\mathbf{Z}})$ be an arbitrary open compact level subgroup unramified away from $N_\Pi$.

  Let $G$ be the group $\GSp_4\times \GL_2$ defined over $\mathbf{Q}$. In \cite{hsu2020eulersystemsmathrmgsp4times}, following \cite{loeffler2021euler}, the authors construct an infinite-level Hecke equivariant symbol map
$$\mathcal{S}_{(0)}(\mathbf{A}_\mathrm{f}^2,\mathbf{Q})\otimes _\mathbf{Q}C_c^\infty(G(\mathbf{A}_\mathrm{f}),\mathbf{Q})\longrightarrow H_\mathrm{mot}^5\left(Y_G,\mathscr{W}_\mathbf{Q}^{a,b,d,*}(3-a-r)\right)[-a-r]$$
where $\mathcal{S}_{(0)}(\mathbf{A}_\mathrm{f}^2,\mathbf{Q})$ denotes a space of Schwartz functions, $Y_G$ is the infinite level Shimura variety for $G$, and $\mathscr{W}_\mathbf{Q}^{a,b,d,*}(3-a-r)$ is an appropriate relative Chow motive associated to an algebraic representation of $G$ parametrized by certain integers $a,b,d,r\in\mathbf{Z}_{\geq 0}$ which satisfy certain inequalities, and which we then fix according to the weight of $\Pi$. 

Let $c\in\mathbf{Z}_{>0}$ be an integer coprime to $6N_\Pi N_{K/\mathbf{Q}} (\mathfrak{f})\Delta_K$  and $\mathfrak{n}\subseteq\mathcal{O}_K$ an integral ideal of $K$ divisible by $\mathfrak{f}$. By appropriately defining local input data on the left-hand side of the symbol map at primes dividing $N_\Pi N_{K/\mathbf{Q}}(\mathfrak{f})\Delta_K$, we construct an integral class in the sense of \cite[\S $3$]{Loeffler_2021}, in 
$$H^5_\mathrm{mot}\left(Y_{\GSp_4}(U^{\GSp_4}(N_\Pi))\times Y_1(N_{K/\mathbf{Q}}(\mathfrak{n})\Delta_K),\mathscr{W}_\mathbf{Q}^{a,b,d,*}(3-a-r)\right).$$

Taking $L$ to be a large enough number field and $H_\mathfrak{n}$ the ray class group of $K$ of modulus $\mathfrak{n}$, we regard this class as an element of a certain Hecke quotient $\mathbf{H}_\mathrm{mot}^5(\Pi,\psi,\mathfrak{n},L)$ defined in \Cref{sec quotients of coh.}, which is naturally a module over $L[H_\mathfrak{n}].$ We call this class $_c\Xi_\mathfrak{n}^{\Pi,\psi}$. Constructing the correct local data defining these classes and proving that they satisfy the required Euler system norm-relations as elements of $\mathbf{H}_\mathrm{mot}^5(\Pi,\psi,\mathfrak{n},L)$ while suitably varying $\mathfrak{n}$, is one of the main difficulties and novelties of this work. It takes up the majority of the first couple of sections, and culminates in \Cref{cor norm-relations over ray class groups}.

The rest of \Cref{sec norm-relations over ray class groups} is concerned with transforming \Cref{cor norm-relations over ray class groups} into a statement in Galois cohomology over ray class fields of $K$. Let $p\nmid N_\Pi N_{K/\mathbf{Q}} (\mathfrak{f})\Delta_K$ and let $\mathfrak{P}$ be a prime of $L$ above $p$. Let $V_\Pi$ be the associated four-dimensional $L_\mathfrak{P}$-linear Galois representation and $V_\Pi^*$ be its dual. Using the $p$-adic \'etale realisation, an Abel-Jacobi map, and a choice of non-zero vector in $\Pi_\mathrm{f}^{U^{\GSp_4}(N_\Pi)}$, we construct a well-defined map from $\mathbf{H}_\mathrm{mot}^5(\Pi,\psi,\mathfrak{n},L)$ to the first Galois cohomology 
\begin{align}\label{eq: intro 2}H^1\left(\mathbf{Q},V_\Pi^*(-k_2+2)\otimes_{L_\mathfrak{P}} H^{1}(\psi,\mathfrak{n},L_\mathfrak{P})\right).\end{align}
We refer the reader to \Cref{def H^1p} for the definition of the groups $H^{1}(\psi,\mathfrak{n},L_\mathfrak{P})$, which are defined as in \cite[\S $5.2$]{Lei_Loeffler_Zerbes_2015}. We call the image of $_c\Xi_\mathfrak{n}^{\Pi,\psi}$ in \eqref{eq: intro 2}, $_c z_\mathfrak{n}^{\Pi,\psi}$. 

Let $\psi_\mathfrak{P}$ be the $p$-adic avatar of the Groessencharacter $\psi$ regarded as a continuous $L_\mathfrak{P}^\times$-valued character of $\mathrm{Gal}(\overline{K}/K)$ via class field theory. In \Cref{sec: CM patching}, we apply the CM-patching method of \cite[\S $5.2$]{Lei_Loeffler_Zerbes_2015} to (non-canonically) identify \eqref{eq: intro 2} with
\begin{align}\label{eq: intro 3}
    H^1\left(K(\mathfrak{n}),V_\Pi^*(-k_2+2)(\psi_\mathfrak{P}^{-1})\right)
\end{align}
where $K(\mathfrak{n})$ denotes the maximal abelian $p$-extension inside the ray class field of $K$ of modulus $\mathfrak{n}.$ We can thus define the class $_c\mathbf{z}_\mathfrak{n}^{\Pi,\psi}$ to be the image of $_cz_\mathfrak{n}^{\Pi,\psi}$ in \eqref{eq: intro 3}. Recall that these classes were defined for $\mathfrak{n}$ which are divisible by $\mathfrak{f}$. Finally, for an integral ideal $\mathfrak{n}\subseteq \mathcal{O}_K$ coprime to $\mathfrak{f}$, we set $_c\mathbf{z}_\mathfrak{n}^{\Pi,\psi}$ to be the image of $_c\mathbf{z}_\mathfrak{nf}^{\Pi,\psi}$ under the Galois corestriction map. These are our Euler system classes.
\begin{theorem}[\Cref{thm split norm-relations in Galois}]\label{thm: intro A}
     Let $\mathscr{A}$ be the set of ideals of $\mathcal{O}_K$ whose norm is coprime to $cpN_\Pi N_{K/\mathbf{Q}}(\mathfrak{f})\Delta_K.$ Then for any $\mathfrak{n,l}\in\mathscr{A}$ with $\ell=\mathfrak{l\overline{l}}\nmid\mathfrak{n}$, a split prime, we have 
$$\mathrm{cores}^{K(\mathfrak{nl})}_{K(\mathfrak{n})}\left( _c \mathbf{z}_\mathfrak{nl}^{\Pi,\psi} \right)=Q_\mathfrak{l}(\sigma_\mathfrak{l}^{-1})\cdot\ _c \mathbf{z}_\mathfrak{n}^{\Pi,\psi}$$
    where $Q_\mathfrak{l}(X)=\det(1-X\ \mathrm{Frob}_\mathfrak{l}^{-1}|V_\Pi(\psi_\mathfrak{P})(k_2-1))$, $\mathrm{Frob}_\mathfrak{l}\in\mathrm{Gal}(\overline{K}/K)$ is arithmetic Frobenius and $\sigma_\mathfrak{l}$ is its image in $\mathrm{Gal}(K(\mathfrak{n})/K).$ Moreoever, every such class lies in the image of the cohomology of a Galois stable $\mathcal{O}_{L_\mathfrak{P}}$-lattice $T_\Pi^*(-k_2+2)(\psi_\mathfrak{P}^{-1})$.
\end{theorem}

 Write $_c\mathbf{z}^{\Pi,\psi}\in H^1(K,V_\Pi^*(-k_2+2)(\psi_\mathfrak{P}^{-1}))$ for the image of the class $_c\mathbf{z}_{1}^{\Pi,\psi}$ (i.e. $\mathfrak{n}=1$) under evaluation at the trivial character of $\mathrm{Gal}(K(1)/K)$. Let $T:=T_\Pi^*(-k_2+2)(\psi_\mathfrak{P}^{-1}), W:=T\otimes_{\mathcal{O}_{L_\mathfrak{P}}}(L_\mathfrak{P}/\mathcal{O}_{L_\mathfrak{P}})$ and $W^*:=T^\vee(1).$ In \Cref{sec: selmer bound} we use \Cref{thm: intro A} directly, together with the classical machinery of Rubin \cite{rubin2000annals}, to obtain a bound on the strict Selmer group $\mathrm{Sel}_{\Sigma_p}(K,T^\vee(1)).$ In this paper, we do this under the following hypotheses:
\begin{itemize}
    \item $\Pi$ has level one (i.e. $\Pi_\ell$ is unramified for all primes $\ell$) and $k_1>k_2$.
    \item $T$ satisfies a ``big image'' assumption in the sense of \cite{rubin2000annals}.
\end{itemize}
\begin{rem}For the precise nature of the ``big image'' hypothesis, the role of the weight hypothesis $k_1>k_2$ and comments on how to extend it to arbitrary level, we refer the reader to \Cref{sec hypotheses}. Moreover, by Chebotarev's density theorem, in order to bound Selmer groups in our setting it is enough to have Euler system norm-relations at split primes of $K$, which is exactly what \Cref{thm: intro A} supplies us with.  
\end{rem}
\begin{theorem}[\Cref{thm selmer bound}]
    Let $\Sigma_p$ be the set of places of $K$ above $p$. Suppose that the class $_c \mathbf{z}^{\Pi,\psi}$ is non-zero and the hypotheses of \emph{\Cref{sec hypotheses}} are satisfied. Then 
$$\mathcal{\ell}_{\mathcal{O}_{L_\mathfrak{P}}}\left(\mathrm{Sel}_{\Sigma_p}(K,T^\vee(1))\right)\leq \mathrm{ind}_{\mathcal{O}_{L_\mathfrak{P}}}(_c \mathbf{z}^{\Pi,\psi})+\mathfrak{n}_W+\mathfrak{n}_W^*<\infty$$
where the three quantities on the right are defined as in \cite[\S $2$]{rubin2000annals}. In particular, the strict Selmer group $\mathrm{Sel}_{\Sigma_p}(K,T^\vee(1))$ is finite.
\end{theorem}

\subsection{Acknowledgments}
I would like to thank my PhD supervisor David Loeffler for bringing this problem to my attention, and for his ongoing guidance and support. I would also like to thank Ananyo Kazi for helpful discussions during our stay in Alpbach.
{
  \hypersetup{linkcolor=black}
  \setcounter{tocdepth}{2}
  \tableofcontents
}
\section{General notation}
\begin{itemize}
    \item We use the following matrix model for the general symplectic group of degree $4$. Let $J$ be the 
$4\times4$ skew-symmetric matrix $\left[\begin{smallmatrix}
        & & & 1\\
        & & 1 &\\
        & -1 & &\\
        -1 & & & 
    \end{smallmatrix}\right]$. We write $\GSp_4$ for the algebraic group over $\mathbf{Z}$ over defined by 
    $$\GSp_4(A):=\left\{g\in \GL_4(A)\ |\ g^t J g=\mu(g) J\right\}$$
 for any unital $\mathbf{Z}$-algebra $A$. The map $\mu: \GSp_4\rightarrow \GL_1, g\mapsto\mu(g)$ is called the multiplier map.
 \item We write $G:=\GSp_4\times \GL_2$ and $H:=\GL_2\times_{\GL_1}\GL_2$ where the fiber product is taken with respect to the determinant map from $\GL_2$ to $\GL_1$.
 \item We always regard $H$ as a subgroup of $G$ via the embedding 
 $$\iota:(\left[\begin{smallmatrix}
     a_1 & b_1\\
     c_1 & d_1
 \end{smallmatrix}\right],\left[\begin{smallmatrix}
     a_2 & b_2\\
     c_2 & d_2
 \end{smallmatrix}\right])\mapsto (\left[\begin{smallmatrix}
     a_1 & & & b_1\\
     & a_2 & b_2 & \\
     & c_2 & d_2 & \\
     c_1 & & & d_1
 \end{smallmatrix}\right],\left[\begin{smallmatrix}
     a_2 & b_2\\
     c_2 & d_2
 \end{smallmatrix}\right]).$$
 In the same manner, we regard $H$ as a subgroup of $\GSp_4$ by projecting $\iota$ to the $\GSp_4$ component.
 \item For a $\mathbf{Z}$-algebra $A$, We write $\mathcal{S}(\mathbf{Q}_\ell^2,A)$ for the space of $A$-valued Schwartz functions on $\mathbf{Q}_\ell^2.$ If $A=\mathbf{C}$ we omit it from the notation. We regard $\mathcal{S}(\mathbf{Q}_\ell^2,A)$ as a smooth left $H(\mathbf{Q}_\ell)$-representation via the first projection map $\mathrm{pr_1}:H\rightarrow \GL_2$. In other words, for $\phi\in \mathcal{S}(\mathbf{Q}_\ell^2)$ and $h=(h_1,h_2)\in H(\mathbf{Q}_\ell)$, we set $h\cdot \phi:=\phi((-)h_1)$. For an open compact subset $V\subseteq\mathbf{Q}_\ell^2$ we write $\ch(V)$ for the characteristic function of $V$.
 \item For $\Gamma$ a linear algebraic group over $\mathbf{Q}_\ell$, with $\Gamma(\mathbf{Q}_\ell)$ unimodular, we write $\mathcal{H}_{\Gamma(\mathbf{Q}_\ell)}^\circ$ for the spherical Hecke algebra of $\Gamma(\mathbf{Z}_\ell)$-bi-invariant compactly supported functions on $\Gamma(\mathbf{Q}_\ell)$, under the usual convolution taken with respect to the normalized Haar measure giving $\Gamma(\mathbf{Z}_\ell)$ volume $1$. For an open compact subset $U\subseteq \Gamma(\mathbf{Q}_\ell)$ we write $\ch(U)$ for the characteristic function of $U$.
 \item For a function $f$ on a group $\mathcal{G}$, we denote by $f'$ the function on $\mathcal{G}$ given by $f'(\gamma):=f(\gamma^{-1}).$ 
\end{itemize}
\section{Local structure theorems and zeta integrals}
For an open compact subgroup $U\subseteq G(\mathbf{Q}_\ell)$, we write $\mathcal{I}(G(\mathbf{Q}_\ell)/U)$ for the space of $H(\mathbf{Q}_\ell)$-coinvariants
 $$\left[\mathcal{S}(\mathbf{Q}_\ell^2)\otimes C_c^\infty(G(\mathbf{Q}_\ell)/U)\right]_{H(\mathbf{Q}_\ell)}$$
 where the action of $H(\mathbf{Q}_\ell)$ on $C_c^\infty(G(\mathbf{Q}_\ell)/U)$ is given by $h\cdot \xi:=\xi(h^{-1}(-))$ for $\xi\in C_c^\infty(G(\mathbf{Q}_\ell)/U).$ In this section, our first goal is to show that the 
$\mathcal{H}_{H(\mathbf{Q}_\ell)}^\circ\otimes\mathcal{H}_{\GSp_4(\mathbf{Q}_\ell)}^\circ$-module $C_c^\infty(H(\mathbf{Z}_\ell)\backslash \GSp_4(\mathbf{Q}_\ell)/\GSp_4(\mathbf{Z}_\ell))$ under the action 
$$((\theta_1\otimes\theta_2)\cdot\xi )(x):=\int_{\GSp_4(\mathbf{Q}_\ell)}\int_{H(\mathbf{Q}_\ell)}\theta_1(h)\theta_2(g)\xi(h^{-1}xg)\ dh\ dg$$
is cyclic, and generated by $\ch(\GSp_4(\mathbf{Z}_\ell)).$ We will then use this to show that the module $\mathcal{I}(G(\mathbf{Q}_\ell)/G(\mathbf{Z}_\ell))$ over the Hecke algebra $\mathcal{H}_{G(\mathbf{Q}_\ell)}^\circ$, is free of rank one generated by $\ch(\mathbf{Z}_\ell^2)\otimes\ch(G(\mathbf{Z}_\ell))$ under the natural right translation action
$$\theta\cdot (\phi\otimes\xi):=\phi\otimes\left(\int_{G(\mathbf{Q}_\ell)}\theta(g)\xi((-)g)\ dg\right).$$
\subsection{Gejima's decomposition} We heavily rely on the results of Gejima \cite{Gejima2018AnEF} which in turn closely follow the approach of Kato-Murase-Sugano \cite{kato2003whittaker} on uniqueness of unramified Shintani functions for the special orthogonal group. We define the following matrix
$$A:=\left[\begin{smallmatrix}
    1 & & & \\
    & 1 & &  \\
    & & & 1\\
    & & 1 & 
\end{smallmatrix}\right]\in\GL_4(\mathbf{Z}_\ell).$$
Gejima uses a slightly different matrix model for the general symplectic group and thus also a sifferent embedding of $H$. Namely, the matrix group he calls $\mathbf{G}$ is related to our $\GSp_4$ through the isomorphism $\mathbf{G}\simeq\GSp_4, g\mapsto AgA$, which identifies our choice of embedding $\iota$ of $H$ with that of \textit{op.cit}. 
\begin{thm}[\cite{Gejima2018AnEF}]\label{thm: cartan decomp}
     There exists a Cartan type decomposition
    $$\GSp_4(\mathbf{Q}_\ell)=\bigsqcup_{\substack{\mu\in\Lambda^+\\ \mu'\in\Lambda_0^{++}}}H(\mathbf{Z}_\ell)\ \mathbf{g}(\mu',\mu)\ \GSp_4(\mathbf{Z}_\ell)$$
    where 
    \begin{itemize}
        \item $\Lambda^+:=\{\mu=(\mu_1,\mu_2,\mu_3)\in\mathbf{Z}^3\ |\ \mu_1\geq \mu_2, 2\mu_2\geq \mu_3\}$.
        \item $\Lambda_0^{++}:=\{\mu'=(\mu'_1,\mu'_2,\mu'_3)\in\mathbf{Z}^3\ |\ \mu'_1\geq 0, 2\mu'_2\geq \mu'_3\}$.
        \item $\mathbf{g}(\mu',\mu):=\mathbf{t}(\mu')\cdot B\cdot\mathbf{t}(\mu)\in \GSp_4(\mathbf{Q}_\ell)$ where for $\nu\in\mathbf{Z}^3$, $\mathbf{t}(\nu):=\mathrm{diag}(\ell^{\nu_1},\ell^{\nu_2},\ell^{\nu_3-\nu_2},\ell^{\nu_3-\nu_1})\in H(\mathbf{Q}_\ell)$ and $B$ is the matrix $A\left[\begin{smallmatrix}
            1 & 1 & & 1 \\
            & 1 & 1 & \\
            & & 1 & \\
            & & -1 & 1
        \end{smallmatrix}\right]A=\left[\begin{smallmatrix}
          1 & 1& 1& \\
           & 1& & 1\\
          & & 1 & -1\\
          & & & 1
        \end{smallmatrix}\right]$.
    \end{itemize}
    \begin{proof}
        This is \cite[Theorem $3.2.1$]{Gejima2018AnEF} after passing through the identification $\mathbf{G}\simeq\GSp_4$ and correcting a minor typographical error in \textit{op.cit}. For clariy, the representatives $\mathbf{g(\mu',\mu)}$ defined here are related to the representatives $g(\mu',\mu)$ of \textit{loc.cit} via $Ag(\mu',\mu)A=\mathbf{g}(\mu',\mu).$ 
    \end{proof}
\end{thm}
Gejima then introduces a partial order $\geq_S$ on the set $S:=\Lambda_0^{++}\times \Lambda^+$ (\cite[Definition $3.3.9$]{Gejima2018AnEF}) and proves the following technical lemma.
\begin{lem}[\cite{Gejima2018AnEF}]\label{lem partial order}
    Let $(\mu',\mu)\in S$.
    \begin{enumerate}
        \item If $(\lambda',\lambda)\in S$ satisfies
        $$H(\mathbf{Z}_\ell)\mathbf{t}(\mu')\GSp_4(\mathbf{Z}_\ell)\mathbf{t}(\mu)\GSp_4(\mathbf{Z}_\ell)\ \cap\ H(\mathbf{Z}_\ell)\mathbf{g}(\lambda',\lambda) \GSp_4(\mathbf{Z}_\ell)\neq \emptyset$$
        then $(\mu',\mu)\geq_S(\lambda',\lambda)$.
        \item The number of $(\lambda',\lambda)\in S$ satisfying $(\mu',\mu)\geq_S(\lambda',\lambda)$ is finite and is denoted by $m(\mu',\mu).$ We have $1\leq m(\mu',\mu)<\infty.$
        \item If $(\lambda',\lambda)\in S$ satisfies $(\mu',\mu)\geq_S(\lambda',\lambda)$ and $(\lambda',\lambda)\neq(\mu',\mu)$, then $m(\lambda',\;\lambda)<m(\mu',\mu)$.
    \end{enumerate}
    \begin{proof}
        This is \cite[Lemma $3.3.10$, Remark $3.3.11$ \& Lemma $3.3.12$]{Gejima2018AnEF} under the identification $\mathbf{G}\simeq \GSp_4$.
    \end{proof}
\end{lem}
\begin{cor}\label{cor: cyclicity ala Gejima}
    The $\mathcal{H}_{H(\mathbf{Q}_\ell)}^\circ\otimes\mathcal{H}_{\GSp_4(\mathbf{Q}_\ell)}^\circ$-module $C_c^\infty(H(\mathbf{Z}_\ell)\backslash \GSp_4(\mathbf{Q}_\ell)/\GSp_4(\mathbf{Z}_\ell))$ is cyclic, generated by $\ch(\GSp_4(\mathbf{Z}_\ell)).$
    \begin{proof}
        For ease of notation, we write $M$ for the module $C_c^\infty(H(\mathbf{Z}_\ell)\backslash \GSp_4(\mathbf{Q}_\ell)/\GSp_4(\mathbf{Z}_\ell))$ and $M'$ for the submodule $(\mathcal{H}_{H(\mathbf{Q}_\ell)}^\circ\otimes\mathcal{H}_{\GSp_4(\mathbf{Q}_\ell)}^\circ)\cdot\xi_0\subseteq M$ where $\xi_0:=\ch(\GSp_4(\mathbf{Z}_\ell))$. We write $\xi_{\mu',\mu}$ for the characteristic function of the double coset $H(\mathbf{Z}_\ell)\ \mathbf{g}(\mu',\mu)\ \GSp_4(\mathbf{Z}_\ell)$. Then by \Cref{thm: cartan decomp} it suffices to show that the $\xi_{\mu',\mu}$ lie in $M'$ for all $(\mu',\mu)\in S$. The proof now proceeds by induction on $m(\mu',\mu)\geq 1$. Let $(\mu',\mu)\in S$. If $\theta_1\in\mathcal{H}_{H(\mathbf{Q}_\ell)}^\circ$ is given by $\ch(H(\mathbf{Z}_\ell)\mathbf{t}(\mu')H(\mathbf{Z}_\ell))$ and $\theta_2\in\mathcal{H}_{\GSp_4(\mathbf{Q}_\ell)}^\circ$ is given by $\ch(\GSp_4(\mathbf{Z}_\ell) \mathbf{t}(\mu)^{-1} \GSp_4(\mathbf{Z}_\ell))$, then the function $(\theta_1\otimes\theta_2)\cdot\xi_0$ is supported on the set $H(\mathbf{Z}_\ell)\mathbf{t}(\mu')\GSp_4(\mathbf{Z}_\ell)\mathbf{t}(\mu)\GSp_4(\mathbf{Z}_\ell)$. Thus, by \Cref{lem partial order} it is a finite linear combination of $\xi_{\lambda',\lambda}$ with $(\mu',\mu)\geq_S (\lambda',\lambda)$. The base case of the induction is $m(\mu',\mu)=1$ in which case parts $(2)$ and $(3)$ of \Cref{lem partial order} imply that $(\theta_1\otimes\theta_2)\cdot\xi_0$ is a non-zero scalar multiple of $\xi_{\mu',\mu}$ and we are done. The inductive step follows in the same manner.
    \end{proof}
\end{cor}
\subsection{Passing to $\mathcal{I}(G(\mathbf{Q}_\ell)/G(\mathbf{Z}_\ell))$} We write $A$ for the rank one torus given by $Z_{\GL_2}$. We have an embedding of Hecke algebras 
$$\Delta:\mathcal{H}_{A(\mathbf{Q}_\ell)}^\circ\longrightarrow\mathcal{H}_{\GL_2(\mathbf{Q}_\ell)}^\circ,\ \ch(\left[\begin{smallmatrix}
    \ell & \\
    & \ell
\end{smallmatrix}\right]A(\mathbf{Z}_\ell))\mapsto\ch(\left[\begin{smallmatrix}
    \ell & \\
    & \ell
\end{smallmatrix}\right]\GL_2(\mathbf{Z}_\ell)).$$
\begin{lem}[\cite{Loeffler_2021}]\label{lem Hecke alg hom}
    There exists a Hecke algebra homomorphism $\zeta:\mathcal{H}_{\GL_2(\mathbf{Q}_\ell)}^\circ\rightarrow \mathcal{H}_{A(\mathbf{Q}_\ell)}^\circ$ such that $\theta\cdot\ch(\mathbf{Z}_\ell^2)=(\Delta\circ\zeta)(\theta)\cdot \ch(\mathbf{Z}_\ell^2)$ for all $\theta\in\mathcal{H}_{\GL_2(\mathbf{Q}_\ell)}^\circ.$
\end{lem}
\begin{proof}
    This is \cite[Lemma $4.3.3$]{Loeffler_2021}
\end{proof}
\begin{thm}\label{thm: cyclicity of I}
    The $\mathcal{H}_{G(\mathbf{Q}_\ell)}^\circ$-module $\mathcal{I}(G(\mathbf{Q}_\ell)/G(\mathbf{Z}_\ell))$ cyclic and generated by $\ch(\mathbf{Z}_\ell^2)\otimes \ch(G(\mathbf{Z}_\ell)).$
    \begin{proof}
        The proof of this is inspired by the proof of \cite[Theorem $4.3.6$]{Loeffler_2021}. However, one needs to be a bit careful at certain points since in this setup, our group $H$ involves two copies of $\GL_2$ rather than one. We firstly note that there is a canonical isomorphism \begin{align}\label{eq: 1}C_c^\infty(G(\mathbf{Q}_\ell)/G(\mathbf{Z}_\ell))&\simeq C_c^\infty(\GSp_4(\mathbf{Q}_\ell)/\GSp_4(\mathbf{Z}_\ell))\otimes C_c^\infty(\GL_2(\mathbf{Q}_\ell)/\GL_2(\mathbf{Z}_\ell))\\
        \nonumber\xi&\mapsto \xi|_{\GSp_4(\mathbf{Q}_\ell)}\otimes \xi|_{\GL_2(\mathbf{Q}_\ell)}.
        \end{align}
        This is clearly equivariant under the Hecke action of $\mathcal{H}_{G(\mathbf{Q}_\ell)}^\circ\simeq \mathcal{H}_{\GSp_4(\mathbf{Q}_\ell)}^\circ\otimes \mathcal{H}_{\GL_2(\mathbf{Q}_\ell)}^\circ$. It is also $H(\mathbf{Q}_\ell)$-equivariant where the $C_c^\infty(\GL_2(\mathbf{Q}_\ell)/\GL_2(\mathbf{Z}_\ell))$-factor on the right of \eqref{eq: 1} is regarded as a $H(\mathbf{Q}_\ell)$-representation in the natural way after projecting to the second factor. Thus, after tensoring with $\mathcal{S}(\mathbf{Q}_\ell^2)$ and taking $H(\mathbf{Q}_\ell)$-coinvariants we obtain a Hecke equivariant isomorphism 
        \begin{align}\label{eq: 2}
            \mathcal{I}(G(\mathbf{Q}_\ell)/G(\mathbf{Z}_\ell))\simeq \left[\mathcal{S}(\mathbf{Q}_\ell^2)\otimes C_c^\infty(\GSp_4(\mathbf{Q}_\ell)/\GSp_4(\mathbf{Z}_\ell))\otimes C_c^\infty(\GL_2(\mathbf{Q}_\ell)/\GL_2(\mathbf{Z}_\ell)) \right]_{H(\mathbf{Q}_\ell)}.
        \end{align}
        We start with an arbitrary element $[\delta]:=[\phi\otimes\xi]\in \mathcal{I}(G(\mathbf{Q}_\ell)/G(\mathbf{Z}_\ell))$ where the square brackets always denote the class under coinvariants. We may assume that $\phi$ is $A(\mathbf{Z}_\ell)$-invariant and thus by the proof of \cite[Proposition $4.3.4$]{Loeffler_2021} we may assume that $\phi=\theta\cdot_H\phi_0$ for some $\theta\in C_c^\infty(H(\mathbf{Q}_\ell)/H(\mathbf{Z}_\ell))$, where $\phi_0:=\ch(\mathbf{Z}_\ell^2)$ and $\cdot_H$ denotes the Hecke action over $H$. Thus $[\delta]=[\phi_0\otimes(\theta'\cdot_H\xi)]$ where recall that $(\theta'\cdot_H\xi)(g)=\int_{H(\mathbf{Q}_\ell)}\theta'(h)\xi(h^{-1}g)\ dh$ for all $g\in G(\mathbf{Q}_\ell).$ Under \eqref{eq: 1} we have $\theta'\cdot_H\xi\simeq\tilde{\xi}_1\otimes\tilde{\xi}_2$ where 
        \begin{align*}
            \tilde{\xi}_1&:=(\theta'\cdot_H\xi)|_{\GSp_4(\mathbf{Q}_\ell)}\in C_c^\infty(H(\mathbf{Z}_\ell)\backslash \GSp_4(\mathbf{Q}_\ell)/\GSp_4(\mathbf{Z}_\ell))\\
            \tilde{\xi}_2&:=(\theta'\cdot_H\xi)|_{ \GL_2(\mathbf{Q}_\ell)}\in\mathcal{H}_{\GL_2(\mathbf{Q}_\ell)}^\circ.
        \end{align*}
        This follows from an unraveling of the action and recalling that we always regard $H$ as a subgroup of $G$ under $\iota$. Hence, under \eqref{eq: 2} we may assume that $[\delta]=[\phi_0\otimes \tilde{\xi}_1\otimes\ch(\GL_2(\mathbf{Z}_\ell))].$ But by \Cref{cor: cyclicity ala Gejima} and lineariy, we can assume that $\tilde{\xi}_1$ is given by $(\theta_1\otimes\theta_2)\cdot\ch(\GSp_4(\mathbf{Z}_\ell))=\theta_1\cdot_H\theta_2'$ with $\theta_1\in\mathcal{H}_{H(\mathbf{Q}_\ell)}^\circ$ and $\theta_2\in\mathcal{H}_{\GSp_4(\mathbf{Q}_\ell)}^\circ$. We have
        \begin{align*}
            [\delta]&=[\phi_0\otimes(\theta_1\cdot_H\theta_2')\otimes\ch(\GL_2(\mathbf{Z}_\ell))]\\
            &\simeq[\theta_1'\cdot_H(\phi_0\otimes\ch(\GL_2(\mathbf{Z}_\ell)))\otimes\theta_2'].
        \end{align*}
        Again by linearity, we may also assume that $\theta_1'=\ch(H(\mathbf{Z}_\ell)hH(\mathbf{Z}_\ell))$ for some $h\in H(\mathbf{Q}_\ell)$. Using Cartan decomposition, it's straight forward to see that $\theta_1'\cdot_H(\phi_0\otimes\ch(\GL_2(\mathbf{Z}_\ell)))$ coincides with $(\theta_3\cdot\phi_0)\otimes\theta_4$ where $\theta_3:=\ch(\GL_2(\mathbf{Z}_\ell)h_1\GL_2(\mathbf{Z}_\ell))$ and $\theta_4:=\ch(\GL_2(\mathbf{Z}_\ell)h_2\GL_2(\mathbf{Z}_\ell))$. Putting this together, we may assume that 
        \begin{align*}[\delta]&=[(\theta_3\cdot\phi_0)\otimes\ch(\GL_2(\mathbf{Z}_\ell))\otimes\ch(\GSp_4(\mathbf{Z}_\ell))],\ \theta_3\in\mathcal{H}_{\GL_2(\mathbf{Q}_\ell)}^\circ\\
        &=[(\Delta(\alpha)\cdot\phi_0)\otimes\ch(\GL_2(\mathbf{Z}_\ell))\otimes\ch(\GSp_4(\mathbf{Z}_\ell))],\ \alpha\in\mathcal{H}_{A(\mathbf{Q}_\ell)}^\circ
        \end{align*}
        where the last equality follows from \Cref{lem Hecke alg hom}. By linearity and the definition of the map $\Delta$, we can assume that $\Delta(\alpha)=\ch(\mathrm{diag}(\ell^i,\ell^i)\GL_2(\mathbf{Z}_\ell))$, $i\in\mathbf{Z}$, in which case we have 
        $$[\delta]=[\phi_0\otimes\ch(\mathrm{diag}(\ell^{-i},\ell^{-i})\GL_2(\mathbf{Z}_\ell))\otimes\ch(\mathrm{diag}(\ell^{-i},\ell^{-i},\ell^{-i},\ell^{-i})\GSp_4(\mathbf{Z}_\ell))].$$
        This finally concludes the proof.
    \end{proof}
\end{thm}
\subsection{The principal-series and Siegel sections}\label{sec: siegel sections}
Let $\psi:\mathbf{Q}_\ell\rightarrow\mathbf{C}^\times$ be the standard additive character of conductor $\mathbf{Z}_\ell$. For $\nu,\mu$ smooth characters of $\mathbf{Q}_\ell^\times$ with $\nu\mu^{-1}\neq |\cdot|^{\pm 1}$. We write $I(\nu,\mu)$ for the irreducible generic principal-series representation of $\GL_2(\mathbf{Q}_\ell)$ given by the normalized parabolic induction of the character $\left[\begin{smallmatrix}
    \nu & \\
    & \mu
\end{smallmatrix}\right]$ of the Borel. Given $\phi\in\mathcal{S}(\mathbf{Q}_\ell^2)$, we set 
$$f^\phi(h;\nu,\mu,s):=\nu(\det(h))|\det(h)|^s\int_{\mathbf{Q}_\ell^\times}\phi((0,1)h)(\nu/\mu)(x)|x|^{2s}\ d^\times x,\ \ h\in\GL_2(\mathbf{Q}_\ell)$$
which converges for $\Re(s)$ large enough. Using this, we consider 
\begin{align*}
W^\phi(h;\nu,\mu,s):=\int_{\mathbf{Q}_\ell}f^\phi(\left[\begin{smallmatrix}
    & 1\\
    -1 & 
\end{smallmatrix}\right]\left[\begin{smallmatrix}
    1 & x\\
    & 1
\end{smallmatrix}\right]h;\nu,\mu,s)\psi(x)\ dx,\ \ h\in\GL_2(\mathbf{Q}_\ell).
\end{align*}
As in \cite{loeffler2023localzetaintegralsgsp4gsp4}, we write $W^\phi(h;\nu,\mu):=W^\phi(h;\nu,\mu,\tfrac{1}{2})$ and $W^\phi(\nu,\mu)$ for the function $W^\phi(-;\nu,\mu)$ on $\GL_2(\mathbf{Q}_\ell).$ Note that even though $f^\phi$ might have poles, $W^\phi$ is entire and there is no $s\in\mathbf{C}$ for which $W^\phi(h;\nu,\mu,s)$ vanishes for all $h$ and $\phi$. The space of functions $W^\phi(\nu,\mu)$ for varying $\phi\in\mathcal
S(\mathbf{Q}_\ell^2)$ is the $\psi^{-1}$-Whittaker model of $I(\nu,\mu)$.

Now let $\chi_1,\chi_2,\chi$ be three smooth characters of $\mathbf{Q}_\ell^\times$ for which $$|\cdot|^{\pm 1}\not\in\{\chi_1,\chi_2,\chi_1\chi_2,\chi_1\chi_2^{-1}\}.$$
We write $I(\chi_1,\chi_2;\chi)$ for the irreducible principal-series representation of $\GSp_4(\mathbf{Q}_\ell)$ given by the normalized parabolic induction of the character 
$$\left[\begin{smallmatrix}
    a & & & \\
    & b & & \\
    & & cb^{-1} & \\
    & & & ca^{-1}
\end{smallmatrix}\right]\mapsto\chi_1(a)\chi_2(b)\chi(c)$$
of the Borel. Every such representation is generic and can be idenitifed with its $\psi$-Whittaker model $\mathcal{W}(\Pi,\psi)$, which is a space of smooth functions $W:\GSp_4(\mathbf{Q}_\ell)\rightarrow\mathbf{C}$ satisying 
$$W(\left[\begin{smallmatrix}
    1 & y & *& * \\
    & 1& x& *\\
    & & 1& -y \\
    & & & 1
\end{smallmatrix}\right]g)=\psi(x+y)W(g),\ g\in\GSp_4(\mathbf{Q}_\ell),\ x,y\in \mathbf{Q}_\ell.$$ 

\subsection{Novodvorsky’s zeta-integral}\label{sec zeta integrals}
Once again, we let $\psi:\mathbf{Q}_\ell\rightarrow\mathbf{C}^\times$ be the standard additive character of conductor $\mathbf{Z}_\ell$. Let $\Pi$ be an irreducible admissible generic representation of $\GSp_4(\mathbf{Q}_\ell)$, and $\pi$ an irreducible admissible generic representation of $\GL_2(\mathbf{Q}_\ell)$. We identify $\Pi$ with its $\psi$-Whittaker model and $\pi$ with its $\psi^{-1}$-Whittaker model.
\begin{defn}[\cite{novodvorsky1979automorphic}] Let $\phi\in\mathcal{S}(\mathbf{Q}_\ell^2), W_\Pi\in\mathcal{W}(\Pi,\psi)$ and $W_\pi\in \mathcal{W}(\pi,\psi^{-1})$. We set
$$Z(\phi,W_\Pi,W_\pi;s):=\int_{Z_{\GL_2}(\mathbf{Q}_\ell)N(\mathbf{Q}_\ell)\backslash H(\mathbf{Q}_\ell)}W_\Pi(h)W_\pi(h_2)f^\phi(h_1;1,\omega_\Pi^{-1}\omega_\pi^{-1},s)\ dh.$$
\end{defn}

\begin{thm}[\cite{novodvorsky1979automorphic}, \cite{soudry1984��}, \cite{loeffler2021higher}] There exists $R<\infty$ independent of $\Pi$ and $\pi$ such that the integral $Z(\phi,W_\Pi,W_{\pi};s)$ converges $\Re(s)>R$ and has meromorphic continuation as a rational function of $\ell^s.$ The $\mathbf{C}$-vector space spanned by $Z(\phi,W_\Pi,W_{\pi};s)$ for varying $(\phi,W_\Pi,W_{\pi})$, is a fractional ideal of $\mathbf{C}[\ell^s,\ell^{-s}]$, containing the constant functions.
\end{thm}
\begin{defn}
    We write $L^\mathrm{Nov}(\Pi\times\pi,s)$ for the unique $L$-factor generating the fractional ideal of $\mathbf{C}[\ell^s,\ell^{-s}]$, spanned by the integrals $Z(\phi,W_\Pi,W_{\pi};s).$ 
\end{defn}
\begin{prop}\label{lem linear form} Let $\Pi,\pi$ be irreducible, unramified, generic representations of $\GSp_4(\mathbf{Q}_\ell),\GL_2(\mathbf{Q}_\ell)$ respectively.
The space $\mathrm{Hom}_{H(\mathbf{Q}_\ell)}(\mathcal{S}(\mathbf{Q}_\ell^2)\otimes \Pi\otimes\pi,\mathbf{1})$ is one-dimensional, with generating element given by $$\mathcal{Z}:\phi\otimes W_\Pi\otimes  W_\pi\mapsto\lim_{s\rightarrow 0}\frac{Z(\phi, W_\Pi, W_\pi;s)}{L^\mathrm{Nov}(\Pi\times\pi,s)}$$
which maps the unramified vector $\ch(\mathbf{Z}_\ell^2)\otimes W_\Pi^\mathrm{sph}\otimes W_\pi^\mathrm{sph}$ to $1$.
\end{prop}
\begin{proof}
    The fact that the linear form $\mathcal{Z}$ satisfies the properties in questions follows from \cite[Theorem $8.9$]{loeffler2021higher} (see also \cite[Proposition $7.10$]{hsu2020eulersystemsmathrmgsp4times}). Uniquness follows from \Cref{thm: cyclicity of I}.
\end{proof}
\begin{lem}\label{lem Nov-poly}
   Let $\Pi$ (resp. $\pi$) be an irreducible unramified principal-series of $\GSp_4(\mathbf{Q}_\ell)$ (resp. $\GL_2(\mathbf{Q}_\ell)$). Write $\Theta_{\Pi\times\pi}$ for the spherical Hecke eigensystem of the $G(\mathbf{Q}_\ell)$-representation $\Pi\times\pi$. There exists a unique (degree eight) polynomial $\mathcal{P}^\mathrm{Nov}_\ell(X)\in\mathcal{H}_{G(\mathbf{Q}_\ell)}^\circ[X]$ such that $\Theta_{\Pi\times\pi}(\mathcal{P}_\ell^\mathrm{Nov})(\ell^{-s})=L^\mathrm{Nov}(\Pi\times\pi,s)^{-1}$ for every such $\Pi\times\pi.$ 
   \begin{proof}
       The dual group of $G$ is given by $G(\mathbf{C})=\GSp_4(\mathbf{C})\times\GL_2(\mathbf{C})$. We write $\hat{T}\subseteq G(\mathbf{C})$ for the direct product of the diagonal tori. Let $\alpha=(\alpha_\Pi,\alpha_\pi)\in\hat{T}/W_G$ be the Satake parameter of $\Pi\times \pi$. If we denote the Satake isomorphism by $\mathcal{S}:\mathcal{H}_{G(\mathbf{Q}_\ell)}^\circ\simeq \mathbf{C}[X^*(\hat{T})]^{W_G}$, then by construction we have
       $\Theta_{\Pi\times\pi}(\theta)=\mathrm{ev}_\alpha(\mathcal{S}(\theta))$
       for all $\theta\in\mathcal{H}_{G(\mathbf{Q}_\ell)}^\circ$. The unramified $L$-parameter associated to $\alpha$ has Artin $L$-factor given by $\det(1-\ell^{-s}\alpha)^{-1}$, and its reciprocal lies in the image of $\mathrm{ev}_\alpha :\mathbf{C}[X^*(\hat{T})]^{W_G}[\ell^{\pm s}]\rightarrow\mathbf{C}[\ell^{\pm s}]$. Thus, it remains to verify that this Artin $L$-factor coincides with the Novodvosrky $L$-factor. Every such $\Pi$ is a theta-lift from a $\mathrm{GSO}_{2,2}$-representation $\tau_1\times\tau_2$ where $\tau_i$ are irreducible unramified $\GL_2$-principal-series, and $\alpha_\Pi$ is the image of $(\alpha_{\tau_1},\alpha_{\tau_2})$ under our embedding $\iota: H\hookrightarrow\GSp_4.$ (see \cite{gan2010thetacorrespondencesgsp4} for more details). Hence, the Artin $L$-factor in question is given by the product of Rankin-Selberg $L$-factors $L(\tau_1\times \pi,s)L(\tau_2\times \pi,s)$. By \cite{soudry1984��}, this coincides with $L^\mathrm{Nov}(\Pi\times\pi,s).$ Finally, uniqueness follows by a density argument in the spectrum of the Hecke algebra.
   \end{proof}
\end{lem}
\begin{rem}
   If $\Pi$ is the unramified principal-series $I(\chi_1,\chi_2;\chi)$, then by \cite{gan2010thetacorrespondencesgsp4}, it is the theta-lift of the product of the two $\GL_2(\mathbf{Q}_\ell)$-principal-series $I(\chi_1\chi_2\chi,\chi)\times I(\chi_1\chi,\chi_2\chi)$ and thus $L^\mathrm{nov}(\Pi\times\pi,s)$ can be written down explicitely as a product of two Rankin-Selberg $L$-factors $L(I(\chi_1\chi_2\chi,\chi)\times\pi,s)L(I(\chi_1\chi,\chi_2\chi)\times\pi,s).$ 
\end{rem}

\begin{prop}
    The $\mathcal{H}_{G(\mathbf{Q}_\ell)}^\circ$-module $\mathcal{I}(G(\mathbf{Q}_\ell)/G(\mathbf{Z}_\ell))$ is free of rank one, generated by $\ch(\mathbf{Z}_\ell^2)\otimes\ch(G(\mathbf{Z}_\ell)).$
    \begin{proof}
        By \Cref{thm: cyclicity of I} it suffices to show that the annihilator of the action of the Hecke algebra on the unramified generating vector, is zero. This can be done using \Cref{lem linear form}.
    \end{proof}
\end{prop}
\begin{defn}\label{def unique operator}
    For $\delta\in \mathcal{I}(G(\mathbf{Q}_\ell)/G(\mathbf{Z}_\ell))$, the Hecke operator $\mathcal{P}_\delta$ is the unique element of $\mathcal{H}_{G(\mathbf{Q}_\ell)}^\circ$ which satisfies $\mathcal{P}_\delta\cdot (\ch(\mathbf{Z}_\ell^2)\otimes\ch(G(\mathbf{Z}_\ell))=\delta$.
\end{defn}

\subsection{Piatetski-Shapiro's zeta-integral and Bessel models}
We still use $\psi$ to denote the standard additive character of $\mathbf{Q}_\ell$ of conductor $\mathbf{Z}_\ell$. Let $\Pi$ be an irreducible admissible representation of $\GSp_4(\mathbf{Q}_\ell)$ and let $\Lambda=(\lambda_1,\lambda_2)$ be a pair of smooth characters of $\mathbf{Q}_\ell^\times$ with $\lambda_1\lambda_2=\omega_\Pi$. A (split) $\Lambda$-Bessel model of $\Pi$ is a $\GSp_4(\mathbf{Q}_\ell)$-invariant subspace isomorphic to $\Pi$, inside the space of functions $B:\GSp_4(\mathbf{Q}_\ell)\rightarrow\mathbf{C}$ satisfying
$$B(\left[\begin{smallmatrix}
    1 & & u& v\\
    & 1 & w & u\\
    & & 1  & & \\
    & & & 1
\end{smallmatrix}\right]\left[\begin{smallmatrix}
    x & & & \\
    & y & & \\
    & & x & \\
    & & & y
\end{smallmatrix}\right]g)=\psi(u)\lambda_1(x)\lambda_2(y)B(g)$$
for all $g\in \GSp_4(\mathbf{Q}_\ell).$ If such a model exists then by \cite[Theorem $6.3.2$(i)]{roberts2016some} it is unique and we denote it by $\mathcal{B}_\Lambda(\Pi)$; the $\Lambda$-Bessel model of $\Pi$. 
\begin{defn}\emph{(\cite{piatetski1997functions})}
    Suppose that $\Pi$ admits a $\Lambda$-Bessel model and let $\nu$ be a smooth character of $\mathbf{Q}_\ell^\times$. For $B\in\mathcal{B}_\Lambda(\Pi)$ and $\phi_1,\phi_2\in\mathcal{S}(\mathbf{Q}_\ell^2)$ we set 
    $$Z(B,\phi_1,\phi_2;\Lambda,\nu,s):=\int_{N_H(\mathbf{Q}_\ell) \backslash H(\mathbf{Q}_\ell)} B(h)\phi_1((0,1)h_1)\phi_2((0,1)h_2)\nu(\det(h))|\det(h)|^{s+\tfrac{1}{2}}\ dh$$
    where $N_H$ denotes the unipotent radical of $H$, consisting of upper unitriangular matrices.
\end{defn}
We know from \cite{piatetski1997functions} that the above zeta-integral converges fro $\Re(s)$ large enough and admits unique meromorphic continuation as a rational function of $\ell^s$. 
\begin{thm}\emph{\cite{rosner2025regular}} The complex vector space 
$$\left\{Z(B,\phi_2,\phi_2;\Lambda,\nu,s)\ |\ B\in\mathcal{B}_\Lambda(\Pi),\ \phi_i\in \mathcal{S}(\mathbf{Q}_\ell^2)\right\}$$
is a fractional ideal of $\mathbf{C}[\ell^s,\ell^{-s}]$ containing $\mathbf{C}$. Its is generated by the degree four spin $L$-factor $L(\Pi,s)$ associated to the Langlands parameter of $\Pi$.
\end{thm}
\subsubsection{Generic representations admit split Bessel models} Let $\Pi$ be an irreducible admissible generic representation of $\GSp_4(\mathbf{Q}_\ell)$. For $\mu$ a smooth character of $\mathbf{Q}_\ell^\times$ and  $W\in\mathcal
W(\Pi)$, Novodvosrky \cite{novodvorsky1979automorphic} defines the integral
$$B(W;\mu,s):=\int_{\mathbf{Q}_\ell^\times}\int_{\mathbf{Q}_\ell}W(\left[\begin{smallmatrix}
    a & & & \\
    & a & & \\
    & x & 1 & \\
    & & & 1
\end{smallmatrix}\right]\left[\begin{smallmatrix}
    1 & & & \\
    & & 1 & \\
    & -1 & & \\
    & & & 1
\end{smallmatrix}\right])|a|^{s-\tfrac{3}{2}}\mu(a)\ dx\ d^\times a.$$
He then shows that it converges for $\Re(s)$ large enough and admits unique meromorphic continuation as a rational function of $\ell^s$. Moreover by \cite{takloo2000functions} we know that the complex vector space generated by the $B(W;\mu,s)$ as $W\in\mathcal{W}(\Pi)$ varies, is a fractional ideal of $\mathbf{C}[\ell^s,\ell^{-s}]$ generated by the spin $L$-factor $L(\Pi\times\mu,s)$. Thus, for any $W\in\mathcal{W}(\Pi)$, we can consider the entire function
    \begin{align}
        \tilde{B}_W(g;\mu,s):=\frac{B(gW;\mu,s)}{L(\Pi\times\mu,s)},\ \ g\in\GSp_4(\mathbf{Q}_\ell)
    \end{align}
In fact for any $s$, the space of functions $\{\tilde{B}_W(-;\mu,s)\ |\ W\in\mathcal{W}(\Pi)\}$ is the $\Lambda$-Bessel model for $\Pi$, where $\Lambda:=(\mu^{-1}|\cdot|^{\tfrac{1}{2}-s},\mu\omega_\Pi|\cdot|^{s-\tfrac{1}{2}})$. This follows from \cite[\S $3.4$]{roberts2016some}.

\subsection{Relating Novdvorsky's and Piatetski-Shapiro's zeta integrals}
    Let $\Pi$ be irreducible admissible generic representations of $\GSp_4(\mathbf{Q}_\ell)$ and let $W_1\in\mathcal{W}(\Pi,\psi)$ and $\phi_1\in\mathcal{S}(\mathbf{Q}_\ell^2)$. Let $\pi=I(\nu,\mu)$ and recall the Whittaker function $W^{\phi_2}(\nu,\mu)\in \mathcal{W}(\pi,\psi^{-1})$ associated to any $\phi_2\in\mathcal{S}(\mathbf{Q}_\ell^2)$.  The ``basic formula'' in \cite[Proposition $6.1$]{loeffler2023localzetaintegralsgsp4gsp4} reads 
    \begin{align}\label{eq: basic formula}
        Z(\phi_1,W_1,W^{\phi_2}(\nu,\mu);s)=L(\Pi\times\mu)Z(\tilde{B}_{W_1},\phi_1,\phi_2;\Lambda,\nu,s)
    \end{align}
    where $\Lambda:=(\mu^{-1}|\cdot|^{\tfrac{1}{2}-s},\omega_\Pi\mu|\cdot|^{s-\tfrac{1}{2}})$ and $\tilde{B}_{W_1}=\tilde{B}_{W_1}(-;\mu,s)\in\mathcal
    B_\Lambda(\Pi)$. This was already used in earlier work of Loeffler-Pilloni-Skinner-Zerbes \cite[\S $8$]{loeffler2021higher}.

\section{Motivic classes for $G$}
\subsection{Input data}\label{sec input for G}
\begin{itemize}
    \item Fix an integer $M\geq 1$.
    \item Let $R\in\mathbf{Z}_{\geq 1}$ be an integer coprime to $M$.
    \item For each $\ell|R$ let  $K_\ell:=K_\ell^{\GSp_4}\times K_\ell^{\GL_2}\subseteq G(\mathbf{Z}_\ell)$ be an open compact level subgroup. For $\ell\nmid R$, we set $K_\ell:= G(\mathbf{Z}_\ell)$. 
    
    \item We set $\phi_M(R):=\otimes_{\ell}\phi_\ell\in\mathcal{S}(\mathbf{A}_\mathrm{f}^2,\mathbf{Z})$ where 
    $$\phi_\ell:=\begin{dcases}
        \ch(\mathbf{Z}_\ell^2),\ &\mathrm{if}\ \ell\nmid MR\\
\ch(\ell\mathbf{Z}_\ell\times(1+\ell\mathbf{Z}_\ell)),\ &\mathrm{if}\ \ell\ |\ M\\
        \phi_\ell\in\mathcal{S}(\mathbf{Q}_\ell^2,\mathbf{Z})\ \mathrm{arbitrary},\ &\mathrm{if}\ \ell|R.
    \end{dcases}$$
    \item Similarly, we set $\xi_M(R):=\otimes_{\ell}\xi_\ell\in C_c^\infty(G(\mathbf{A}_\mathrm{f}))$ where 
    $$\xi_\ell:=\begin{dcases}
    \ch(G(\mathbf{Z}_\ell)),\ &\mathrm{if}\ \ell\nmid MR\\
        (\ell^2-1)\xi_0^{\mathcal{K}_{1,0}}(\ell),\ &\mathrm{if}\ \ell|M\\
        C_\ell \cdot\xi_\ell\in C_c^\infty(G(\mathbf{Q}_\ell)/K_\ell,\mathbf{Z})\ \mathrm{arbitrary},\ &\mathrm{if}\ \ell|R
    \end{dcases}$$
    where $C_\ell\in\mathbf{Z}$ is such that $C_\ell \phi_\ell\otimes\xi_\ell$ is integral of level $K_\ell$ for all $\ell|R$ in the sense of \cite[\S $3$]{Loeffler_2021}. Moreover,  $\xi_0^{\mathcal{K}_{0,1}}(\ell)$ is as in \cite[\S $8.1$] {hsu2020eulersystemsmathrmgsp4times} but with $K_{1,0}$ in \textit{loc.cit}, replaced by $\mathcal{K}_{1,0}:=K_{1,0}^{\GSp_4}\times \GL_2(\mathbf{Z}_\ell)$ where $K_{1,0}^{\GSp_4}:=\{g\in \GSp_4(\mathbf{Z}_\ell)\ |\ \mu(g)\equiv 1\mod \ell\}$. This minor change in the setup arises from the fact that the classes constructed in \textit{op.cit} live in the cohomology of the Shimura variety of the smaller group $\GSp_4\times_{\GL_1}\GL_2$. Here we naturally view them as classes in the cohomology of the Shimura variety of $G$. The tame norm-relations and their local counterpart (\cite[Theorem $7.31$]{hsu2020eulersystemsmathrmgsp4times}) carry over, as the identity maps induce isomorphisms of quotient groups 
    \begin{align*}
        (\GSp_4\times_{\GL_1} \GL_2)(\mathbf{Z}_\ell)/K_{1,0}&\simeq G(\mathbf{Z}_\ell)/\mathcal{K}_{1,0} 
    \end{align*}and the Hecke actions are compatible.
    \item In the same spirit, we define the open compact subgroup $K_M(R):=\prod_\ell K_\ell\subseteq G(\mathbf{A}_\mathrm{f})$ where 
    $$K_\ell:=\begin{dcases}
G(\mathbf{Z}_\ell)=\GSp_4(\mathbf{Z}_\ell)\times \GL_2(\mathbf{Z}_\ell),\ &\mathrm{if}\ \ell\nmid MR\\
\mathcal{K}_{1,0}=K_{1,0}^{\GSp_4}\times \GL_2(\mathbf{Z}_\ell),\ &\mathrm{if}\ \ell\ |\ M\\
        K_\ell,\ &\mathrm{if}\ \ell\ |\ R.\\
    \end{dcases}$$
    The subgroup $K_M(R)$ can naturally be written as $K^{
    \GSp_4}_M(R)\times K^{\GL_2}_M(R)$ in the obvious way. Finally, we set $U(R):=\left\{\prod_{\ell|R}K_\ell\right\}\times\left\{\prod_{\ell\nmid R} G(\mathbf{Z}_\ell)\right\}\supseteq K_M(R)$ and the same direct product decomposition $U(R)=U^{\GSp_4}(R)\times U^{\GL_2}(R)$ holds.\\
\end{itemize}

The elements $\phi_M(R)\otimes\xi_M(R)$ are contained in $\mathcal{S}(\mathbf{A}_\mathrm{f}^2,\mathbf{Z})\otimes C_c^\infty(G(\mathbf{A}_\mathrm{f})/K_M(R),\mathbf{Z})$, and they are integral of level $K_M(R)$ in the sense of \cite[\S $3$]{Loeffler_2021}. Let $a,b,d,r\in\mathbf{Z}_{\geq 0}$ satisfying $d\leq a+b$ and $\max\{0,-a+d\}\leq r\leq \min\{b,d\}$. As in \cite{hsu2020eulersystemsmathrmgsp4times}, but in a slightly different notation, there exists a symbol map
$$\mathrm{Symbl}^{[a,b,d,r]}:\mathcal{S}_{(0)}(\mathbf{A}_\mathrm{f}^2,\mathbf{Q})\otimes _\mathbf{Q}C_c^\infty(G(\mathbf{A}_\mathrm{f}),\mathbf{Q})\longrightarrow H_\mathrm{mot}^5(Y_G,\mathscr{W}_\mathbf{Q}^{a,b,d,*}(3-a-r))[-a-r]$$
which is $(H(\mathbf{A}_\mathrm{f})\times G(\mathbf{A}_\mathrm{f}))$-equivariant in the sense of \cite{Loeffler_2021}. The notation $\mathcal{S}_{(0)}$ means that we are considering the full space of Schwartz functions, unless $a=r=0$, in which case we consider the subspace of functions that vanish at the origin. Here $Y_G=Y_{\GSp_4}\times Y_{\GL_2}$ denotes the infinite level Shimura variety of $G$ and $\mathscr{W}_\mathbf{Q}^{a,b,d}$ is the $G(\mathbf{A}_\mathrm{f})$-equivariant relative Chow motive over $Y_G$ appearing in \cite[\S $4.2$]{hsu2020eulersystemsmathrmgsp4times}. The square brackets denote a twist of the $G(\mathbf{A}_\mathrm{f})$-action by $|\mu(-)|_\mathbf{A}^{-a-r}$, where the multiplier map $\mu$ is regarded as a character of $G$ through the $\GSp_4$-factor.
\begin{defn}[\cite{hsu2020eulersystemsmathrmgsp4times}]
    For integers $R,M$ as before, we define the class
    $$z_{\mathrm{mot},M}(R):=\mathrm{Symbl}^{[a,b,d,r]}\left(\phi_{M}(R)\otimes\xi_{M}(R)\right)\in H_\mathrm{mot}^5(Y_G(K_M(R)),\mathscr{W}_\mathbf{Q}^{a,b,d,*}(3-a-r))[-a-r].$$
\end{defn}
\begin{defn}
    Let $c\in\mathbf{Z}_{>1}$ be an integer coprime to $6$ and write $\varpi_c:=\otimes_{\ell|c}\varpi_\ell$ where $\varpi_\ell$ is a uniformizer of $\mathbf{Q}_\ell$. For integers $R,M$ as before satisfying the extra condition $\mathrm{gcd}(c,R)=\mathrm{gcd}(c,M)=1$, we define the integral variant
    $$_cz_{\mathrm{mot},M}(R):=\left\{c^2-(\left[\begin{smallmatrix}
        \varpi_c & & & \\
        & \varpi_c & & \\
        & & \varpi_c & \\
        & & & \varpi_c
    \end{smallmatrix}\right],\left[\begin{smallmatrix}
        \varpi_c & \\
        & \varpi_c
    \end{smallmatrix}\right])^{-1}\right\}\cdot z_{\mathrm{mot},M}(R)$$
    in $H_\mathrm{mot}^5(Y_G(K_M(R)),\mathscr{W}_\mathbf{Q}^{a,b,d,*}(3-a-r))[-a-r]$.
\end{defn}

Let $p$ be a prime. Then for every $c,M,R$ as above with $(M,p)=(R,p)=1$, the classes $_cz_{\mathrm{mot},M}(R)$ under the $p$-adic \'etale regulator map, lie in the corresponding $p$-adic \'etale cohomology with integral coefficients. This follows by construction and part of \cite[Theorem $8.7$]{hsu2020eulersystemsmathrmgsp4times} (see also \cite[Proposition $9.5.2$]{Loeffler_2021}).

We have an isomorphism $Y_G(K_M(R))\simeq Y_G(U(R))\times_\mathbf{Q} \mathbf{Q}(\mu_{M})$ induced by the inclusion $K_M(R)\subseteq U(R).$ Pushing forward the $z$-classes under this isomorphism, we obtain the following two variants, defined over cyclotomic fields:
\begin{align*}
    \Xi_{\mathrm{mot},M}(R)&\in H^5_\mathrm{mot}(Y_{G}(U(R))\times_\mathbf{Q}\mathbf{Q}(\mu_{M}),\mathscr{W}_\mathbf{Q}^{a,b,d,*}(3-a-r))[-a-r]\\
    _c\Xi_{\mathrm{mot},M}(R)&\in H^5_\mathrm{mot}(Y_{G}(U(R))\times_\mathbf{Q}\mathbf{Q}(\mu_{M}),\mathscr{W}_\mathbf{Q}^{a,b,d,*}(3-a-r))[-a-r]
\end{align*}
where the classes $_c\Xi_{\mathrm{mot},M}(R)$ satisfy the same $p$-adic integrality property in \'etale cohomology.

\subsubsection{Motivic $\GSp_4\times \GL_2$ tame norm-relations} We now showcase that the tame norm-relations of \cite{hsu2020eulersystemsmathrmgsp4times} in Galois cohomology, can be strengthened to hold in motivic cohomology. Once, again let $c\in\mathbf{Z}_{>1}$ coprime to $6$. Let $R$ be an integer as above and let $\mathscr{S}$ be the set of square-free positive integers coprime to $cR$
\begin{lem}\label{lem intertwining action}
    Let $\ell,M\in\mathscr{S}$ with $\ell$ prime and $\ell\nmid M$. The map $Y_G(K_M(R))\simeq Y_G(U(R))\times_\mathbf{Q} \mathbf{Q}(\mu_{M})$ intertwines the action of $\mathcal{P}_\ell^{\mathrm{Nov}'}(1)$ on the source with $\mathcal{P}_\ell^\mathrm{Nov'}(\mathrm{Frob}_\ell^{-1})$ on the target, where $\mathrm{Frob}_\ell$ denotes arithmetic Frobenius as an element of $\mathrm{Gal}(\mathbf{Q}(\mu_{M})/\mathbf{Q}).$ 
    \begin{proof}
        By construction, given $t=(t_1,t_2)\in T(\mathbf{Q}_\ell)=T^\mathrm{GSp_4}(\mathbf{Q}_\ell)\times T^{\GL_2}(\mathbf{Q}_\ell)$, the map $Y_G(K_M(R))\simeq Y_G(U(R))\times_\mathbf{Q} \mathbf{Q}(\mu_{M})$ intertwines the action of $\ch(G(\mathbf{Z}_\ell) tG(\mathbf{Z}_\ell))$ on the source, (regarded as an element of the abstract Hecke algebra at $\ell$) with $\ch(G(\mathbf{Z}_\ell) tG(\mathbf{Z}_\ell))\times \mathrm{Frob}_\ell^{v_\ell(\mu(t_1))}$ where recall that $\mu$ is the multiplier map. Thus it suffices to show that for every monomial of the form $\{\prod_j\ch(G(\mathbf{Z}_\ell) t^{(i,j)}G(\mathbf{Z}_\ell))\}X^i$ (with $t^{(i,j)}\in T(\mathbf{Q}_\ell)$ $i\in \{0,\dots, 8\}$) appearing in $\mathcal{P}_\ell^\mathrm{Nov'}(X)$, we have $-i=\sum_jv_\ell(\mu(t_1^{(i,j)})).$ We claim that this is indeed true.
        
        Let $\Pi$ be an irreducible unramified principal-series of $\GSp_4(\mathbf{Q}_\ell)$, and let $\pi=I(|\cdot|^a,|\cdot|^b)$ be an irreducible unramified principal-series of $\GL_2(\mathbf{Q}_\ell)$. If we let $\Pi_1:=\Pi\otimes\omega_\Pi^{-1/2}$ and $\pi_1:=\pi\otimes\omega_\Pi^{1/2}$, then $\Pi_1$ has trivial central character and $L^\mathrm{Nov}(\Pi\times\pi,s)=L^\mathrm{Nov}(\Pi_1\times\pi_1,s)$. We have $L^\mathrm{Nov}(\Pi\times\pi,s)=L(\Pi_1,s+\tilde{a})L(\Pi_1,s+\tilde{b})$ where $\tilde{a}:=a+\tfrac{c}{2}$ and $\tilde{b}=b+\tfrac{c}{2}$. Let
        \begin{align}\label{eq: Hecke ops}\scalemath{0.9}{\mathscr{T}_\ell:=\ch(\GSp_4(\mathbf{Z}_\ell)\left[\begin{smallmatrix}
            \ell & & & \\
            & \ell & & \\
            & & 1 & \\
            & & & 1
\end{smallmatrix}\right]\GSp_4(\mathbf{Z}_\ell)),\ \mathscr{R}_\ell:=\ch(\GSp_4(\mathbf{Z}_\ell)\left[\begin{smallmatrix}
            \ell^2 & & & \\
            & \ell & & \\
            & & \ell & \\
            & & & 1
\end{smallmatrix}\right]\GSp_4(\mathbf{Z}_\ell)),\ \mathscr{S}_\ell:=\ch(\left[\begin{smallmatrix}
            \ell & & & \\
            & \ell & & \\
            & & \ell & \\
            & & & \ell 
\end{smallmatrix}\right]\GSp_4(\mathbf{Z}_\ell))}.\end{align}
        Then, by \cite[Theorem $7.5.3$]{roberts2007local}, we have 
        $$L^\mathrm{Nov}(\Pi\times\pi,s)^{-1}=\prod_{i\in\{\tilde{a},\tilde{b}\}}\left(1- \ell^{-3/2}\lambda_{\Pi_1}\ell^{-s-i}+(\ell^{-2}\mu_{\Pi_1}+1+\ell^{-2})\ell^{-2s-2i}-\ell^{-3/2}\lambda_{\Pi_1}\ell^{-3s-3i}+\ell^{-4s-4i}\right)$$
        where $\lambda_{\Pi_1}$ is the Hecke eigenvalue of $\mathscr{T}_\ell$ acting on $W_{\Pi_1}^\mathrm{sph}$ and similarly $\mu_{\Pi_1}$ is the eigenvalue of $\mathscr{R}_\ell$. It follows by construction that $\lambda_\Pi=\ell^{-\tfrac{c}{2}}\lambda_{\Pi_1}$ and $\mu_\Pi=\ell^{-c}\mu_{\Pi_1}$. Thus, the above product becomes 
        \begin{align}\label{eq: 3}
            \prod_{i\in\{a,b\}}\left(1- \ell^{-3/2}\lambda_{\Pi}\ell^{-s-i}+(\ell^{-2}\mu_{\Pi}+(1+\ell^{-2})\omega_\Pi(\ell))\ell^{-2s-2i}-\ell^{-3/2}\omega_\Pi(\ell)\lambda_{\Pi}\ell^{-3s-3i}+\omega_\Pi(\ell)^2\ell^{-4s-4i}\right)
        \end{align}
        where ofcourse, $\omega_\Pi(\ell)$ is also the $\mathcal{S}_\ell$-eigenvalue of $W_\Pi^\mathrm{sph}$ and $\ell^{-i}$ for $i\in\{a,b\}$ are the Satake parameters of $\pi$. This expression uniquely determines the polynomial $\mathcal{P}_\ell^\mathrm{Nov}(X)$, and hence also $\mathcal{P}_\ell^{\mathrm{Nov}'}(X)$. From here, one can read off the coefficients of each of the $X^i$ and numerically verify the claim, using the multiplier value assigned to each of the matrices defining $\mathscr{T}_\ell,\mathscr{R}_\ell,\mathscr{S}$, which is $\ell,\ell^2,\ell^2$ respectively. For example, under the natural identification $\mathcal{H}_{G(\mathbf{Q}_\ell)}^\circ\simeq\mathcal{H}_{\GSp_4(\mathbf{Q}_\ell)}^\circ\otimes\mathcal{H}_{\GL_2(\mathbf{Q}_\ell)}^\circ,$ the $\GSp_4$-part of the coefficient of $X^5$ in $\mathcal{P}_\ell^\mathrm{Nov}(X)'$ is $\mathscr{S}_\ell'\mathscr{T}^{'}_\ell\mathscr{R}^{'}_\ell+{\mathscr{S}_\ell'}^2\mathscr{T}^{'}_\ell$ as per the claim. All other verifications are handled in the same way.
    \end{proof}
\end{lem}
\begin{thm}[Motivic tame norm-relations]\label{thm motivic tame norm-relations} For all $M,M'\in\mathscr{S}$ with $\tfrac{M'}{M}=\ell$ prime, we have 
\begin{align*}
    \mathrm{norm}^{\mathbf{Q}(\mu_{M'})}_{\mathbf{Q}(\mu_{M})}(_c\Xi_{\mathrm{mot},M'}(R))=\mathcal{P}_\ell^{\mathrm{Nov}'}(\mathrm{Frob}_\ell^{-1})\cdot\  _c\Xi_{\mathrm{mot},M}(R)
\end{align*}
where $\mathrm{Frob}_\ell$ denotes arithmetic Frobenius at $\ell$ as an element of $\mathrm{Gal}(\mathbf{Q}(\mu_{M})/\mathbf{Q}).$
\begin{proof}
    By construction, the input data between the $M'$-class and the $M$-class only differ at $\ell$. The local input data at $\ell$ is $\delta_1:=\phi_{1,1}\otimes(\ell^2-1)\xi_0^{\mathcal{K}_{0,1}}(\ell)$ and $\phi_{0,0}\otimes\ch(G(\mathbf{Z}_\ell))$ respectively, where $\phi_{t,t}:=\ch(\ell^t\mathbf{Z}_\ell\times(1+\ell^t\mathbf{Z}_\ell)).$ Using \Cref{thm: cyclicity of I}, \Cref{lem intertwining action} and the same reduction argument as the proof of \cite[Theorem $9.4.3$]{Loeffler_2021}, it suffices to show that $\mathcal{P}_{\mathrm{Tr}(\delta_1)}'=\mathcal{P}_\ell^\mathrm{Nov}(1)$ where $\mathcal{P}_{\delta_1}$ is as in \Cref{def unique operator} and $\mathrm{Tr}$ denotes the trace map from $\mathcal{K}_{0,1}$ to $G(\mathbf{Z}_\ell).$ Let $\sigma:=\Pi\times\pi$ be as before. The map $\mathfrak{Z}:\mathcal{S}(\mathbf{Q}_\ell^2)\otimes C_c^\infty(G(\mathbf{Q}_\ell))\rightarrow \sigma^\vee$ given by $\phi\otimes\xi\mapsto \left\{W_\Pi\otimes W_\pi \mapsto \mathcal{Z}(\phi\otimes \xi\cdot (W_\Pi\otimes W_\pi))\right\}$ is $H(\mathbf{Q}_\ell)$-equivariant and moreover we can apply \cite[Theorem $7.31$]{hsu2020eulersystemsmathrmgsp4times} to this map. Setting $\xi_1:=\mathrm{Tr}(\xi_0^{\mathcal{K}_{0,1}}(\ell))\in C_c^\infty(G(\mathbf{Q}_\ell)/G(\mathbf{Z}_\ell))$ and evaluating at the spherical vector, \textit{loc.cit} gives 
    \begin{align*}
        \mathcal{Z}\left(\phi_{1,1}\otimes(\ell^2-1)(\xi_1\cdot (W_\Pi^\mathrm{sph}\otimes W_\pi^\mathrm{sph}))\right)&=L^\mathrm{Nov}(\Pi\times\pi,0)^{-1} \mathcal{Z}(\phi_{0,0}\otimes W_\Pi^\mathrm{sph}\otimes W_\pi^\mathrm{sph})\\
        &=L^\mathrm{Nov}(\Pi\times\pi,0)^{-1}\\
        &=\Theta_{\Pi\times\pi}(\mathcal P^\mathrm{Nov}_\ell(1))
    \end{align*}
    where the second equality follows from \Cref{lem linear form} and the third from \Cref{lem Nov-poly}. But from the cyclicity of the Hecke module $\mathcal{I}(G(\mathbf{Q}_\ell)/G(\mathbf{Z}_\ell))$, the $H(\mathbf{Q}_\ell)$-equivariance of $\mathcal
    Z$, and the definition of $\delta_1$, the expression on the left, can also be written as $\Theta_{\Pi\times\pi}(\mathcal{P}_{\mathrm{Tr}(\delta_1)}')$. The result then follows from a density argument. 
\end{proof}
\end{thm}
\Cref{thm motivic tame norm-relations} extends the tame norme-relations of \cite{hsu2020eulersystemsmathrmgsp4times} from a statement in Galois cohomology to a statement in motivic cogomology, and it also covers the boundary case $a=b=d=0$ which was omitted in \textit{op.cit}.
\subsection{Changing the level on the $\GL_2$-factor}\label{sec changing GL2 lvl}
In this section we always keep $M$ fixed. To shorten notation, we write $\mathbf{Q}_{M}$ for $\mathbf{Q}(\mu_{M})$. By $\mathrm{Frob}_\ell$ we will always mean arithmetic Frobenius at $\ell$ as an element of $\mathrm{Gal}(\mathbf{Q}_{M}/\mathbf{Q})$ for any $\ell\nmid M$. We will first specify the input data at $R$, which was taken to be arbitrary in the previous section. 

We fix  $N\in\mathbf{Z}_{\geq 1}$ squarefree and coprime to $c$. We also fix a prime $p$ such that $p\nmid cN$.  For each $\ell|pN$ we fix a $\mathbf{Z}_p$-integral input data $C_\ell\phi_\ell\otimes\xi_\ell$ at some level $K_\ell=K_\ell^{\GSp_4}\times K_\ell^{\GL_2}\subseteq G(\mathbf{Z}_\ell)$. These remain arbitrary.
\begin{rem}
    The integer $N$ will later take the form of the level of a $\GSp_4$ automorphic representation, and the prime $p$ will take form of the residue characteristic of the coefficient field of the associated Galois representation.
\end{rem}
 \begin{defn}\label{def input at R} Let $R\in\mathbf{Z}_{\geq 1}$ coprime to $c$. We define local data at each prime $\ell|R$ as follows. Let ${\ell_0}\nmid pN$ be some auxiliary prime. Fix a smooth unramified character $\mu=\mu_{\ell_0}:\mathbf{Q}_{{\ell_0}}^\times\rightarrow \overline{\mathbf{Z}}.$
 \begin{itemize}\item If $\ell\nmid pN$, we set:
     \begin{enumerate}
     \item $\phi_\ell:=\phi_{\ell,2}:=\ch(\ell^2\mathbf{Z}_\ell\times(1+\ell^2\mathbf{Z}_\ell))$.
\item $K_\ell^{\GSp_4}:=\GSp_4(\mathbf{Z}_\ell)$ and $K_\ell^{\GL_2}:=K_1(\ell^e):=\{k\in\GL_2(\mathbf{Z}_\ell)\ |\ g\equiv \left[\begin{smallmatrix}
    * & * \\
    0 &  1
\end{smallmatrix}\right]\mod \ell^e\}$ where $\ell^e\parallel R$.
         \item  Let $\eta_\ell:=\left[\begin{smallmatrix}
             1 & & \ell^{-1} & \\
             & 1& & \ell ^{-1}\\
              & & 1& \\
              & & & 1
         \end{smallmatrix}\right]$. Then $\xi_\ell:=\xi_{\ell,\mu}:=\xi_\ell^{\GSp_4}\otimes \xi_{\ell,\mu}^{\GL_2}\in C_c^\infty(G(\mathbf{Q}_\ell)/K_\ell^{\GSp_4}\times K_\ell^{\GL_2})$ where
 \begin{align*}\xi_\ell^{\GSp_4}&:=\ch(\GSp_4(\mathbf{Z}_\ell))-\ch(\eta_\ell \GSp_4(\mathbf{Z}_\ell))\\
\xi_{\ell,\mu}^{\GL_2}&:=\begin{dcases} \ell^{-2}(\ell-1)^{-1}\left(\mu(\ell)^{-2}\ch(\left[\begin{smallmatrix}
             & -1 \\
              \ell^2 &
         \end{smallmatrix}\right] K_\ell^{\GL_2})-\ell^{-1/2}\mu(\ell)^{-3}\ch(\left[\begin{smallmatrix}
             & -\ell \\
              \ell^2 &
         \end{smallmatrix}\right] K_\ell^{\GL_2})\right),\ &\ell={\ell_0}\\
         \ell^{-1}(\ell-1)^{-1} \ch(\left[\begin{smallmatrix}
             & -1 \\
              \ell^2 &
         \end{smallmatrix}\right] K_\ell^{\GL_2}),\ &\ell\neq{\ell_0}.
         \end{dcases}\end{align*}
         \item $C_\ell:=\ell^3(\ell-1)^3(\ell+1)^2.$
     \end{enumerate}
     \item If $\ell|pN$, then the local data is the fixed choice $C_\ell\phi_\ell\otimes \xi_\ell$ as above.
     \end{itemize}
 \end{defn}

\noindent As in the previous section, but with this choice of data at $R$, we obtain for each such $R\in\mathbf{Z}_{\geq 1}$ a class $$_c\Xi_{\mathrm{mot}}(\mu[R])\in H_\mathrm{mot}^5(Y_G(U(R))\times \mathbf{Q}_{M},\mathscr{W}_\mathbf{Q}^{a,b,d,*}(3-a-r))[-a-r]_{\overline{\mathbf{Q}}}$$
 where we have dropped the dependence on $M$ in the notation of the classes, and we have added the notation $\mu[R]$. This is to signify that the class $_c\Xi_{\mathrm{mot}}(\mu[R])$ also depends on the integer $R$ and $\mu$.
 \begin{defn}\label{def classes Xi_mot}
    Let $N'\in\mathbf{Z}_{\geq 5}$ coprime to $cpN$, we define the class
     \begin{align*}_c\Xi_{\mathrm{mot}}(\mu[N,N'])&:=\mathrm{pr}_{(N,N'),*}\left(_c\Xi_{\mathrm{mot}}(\mu[NN'])\right)\\
     &\in H^5_\mathrm{mot}(Y_{\GSp_4}(U^{\GSp_4}(N))\times Y_{\GL_2}(U^{\GL_2}(N'))\times \mathbf{Q}_{M},\mathscr{W}_\mathbf{Q}^{a,b,d,*}(3-a-r))[-a-r]_{\overline{\mathbf{Q}}}\end{align*}
     where $\mathrm{pr}_{(N,N')}$ is the natural projection induced by the inclusions $$U^{\GSp_4}(NN')\subseteq U^{\GSp_4}(N)\ \ \mathrm{and}\ \ U^{\GL_2}(NN')\subseteq U^{\GL_2}(N').$$
 \end{defn} 
\noindent  To shorten notation we set 
 $$H_\mathrm{mot}^5(N,N'):=H^5_\mathrm{mot}(Y_{\GSp_4}(U^{\GSp_4}(N))\times Y_{\GL_2}(U^{\GL_2}(N'))\times \mathbf{Q}_{M},\mathscr{W}_\mathbf{Q}^{a,b,d,*}(3-a-r))[-a-r].$$
 Note that by construction, we now have $U^{\GL_2}(N')=U_1(N'):=\{\prod_{\ell^e\parallel N'} K_1(\ell^e)\}\times\{\prod_{\ell\nmid N'} \GL_2(\mathbf{Z}_\ell)\}.$
 
 We write $\mathcal{P}_\ell^\mathrm{spin}(X)\in\mathcal{H}_{\GSp_4(\mathbf{Q}_\ell)}^\circ[X]$ for the degree four polynomial interpolating (spin) $L$-factors $L(\Pi,s)$ of unramified $\GSp_4(\mathbf{Q}_\ell)$-representations $\Pi$. It can be explicetly recovered from the bracketed expression in \eqref{eq: 3} with $i=0$. Note that $\mathcal{P}_\ell^\mathrm{spin}(\ell^{-1/2}X)$ is a polynomial over the $\mathbf{Z}[1/\ell]$-integral Hecke algebra. Finally, we write $\mathrm{pr}_1$ for the canonical projection 
 $$Y_{\GL_2}(U^{\GL_2}(\ell N'))\rightarrow Y_{\GL_2}(U^{\GL_2}( N'))$$
 induced by the inclussion $U^{\GL_2}(\ell N')\subseteq U^{\GL_2}( N').$
 This induces a pushforward map $H_\mathrm{mot}^5(N,\ell N')\rightarrow H_\mathrm{mot}^5(N,N')$ which we denote by $(\mathrm{pr}_1)_*.$ For a prime $\ell$, we once again set $\mathcal{S}_\ell:=\ch(\left[\begin{smallmatrix}
     \ell & \\
     & \ell
 \end{smallmatrix}\right]\GL_2(\mathbf{Z}_\ell))\in\mathcal{H}_{\GL_2(\mathbf{Q}_\ell)}^\circ$ and $\mathcal{T}_\ell:=\ch(\GL_2(\mathbf{Z}_\ell)\left[\begin{smallmatrix}
     \ell & \\
     & 1
 \end{smallmatrix}\right]\GL_2(\mathbf{Z}_\ell))\in \mathcal{H}_{\GL_2(\mathbf{Q}_\ell)}^\circ.$
 \begin{thm}\label{thm gl_2 norm-relations}
        Let $\pi=I(\nu,\mu)$ for some smooth unramified character $\nu$ of $\mathbf{Q}_{\ell_0}^\times$. Let $a_{\ell_0},b_{\ell_0}$ be the spherical $\mathcal{T}_{\ell_0}',\mathcal{S}_{\ell_0}'$ Hecke eigenavlues of $\pi$. If ${\ell_0}\nmid cNN'$ we have 
 \begin{align}\label{eq: 4}
     (\mathrm{pr}_1)_*\left( _c\Xi_{\mathrm{mot}}(\mu[N,{\ell_0} N'])\right)= {\mathcal{P}_{\ell_0}^{\mathrm{spin}}}'(\nu({\ell_0})\mathrm{Frob}_{\ell_0}^{-1})\cdot \ _c\Xi_{\mathrm{mot}}(\mu[N,N'])
 \end{align}
 in the quotient $\left[H_\mathrm{mot}^5(N,N')_{\overline{\mathbf{Q}}}\right]_{\substack{\mathcal{T}_{\ell_0}'=a_{\ell_0}\\
 \mathcal{S}_{\ell_0}'=b_{\ell_0}}}.$
 \end{thm}
 \begin{proof}
     It clearly suffices to prove this after base change to $\mathbf{C}$, and we shall do so implictely. We first note that the classes $_c\Xi_{\mathrm{mot}}(\mu[N,{\ell_0} N'])$ and $_c\Xi_{\mathrm{mot}}(\mu[N, N'])$ differ only at ${\ell_0}$. Thus, in the same manner as the proof of \cite[Theorem 9.4.3]{Loeffler_2021}, fixing the data away from ${\ell_0}$ used to define the two classes, we get a commutative diagram at ${\ell_0}$
     \[\begin{tikzcd}[ampersand replacement=\&]
	{\mathcal{I}(G(\mathbf{Q}_{\ell_0})/U)} \&\& {H_\mathrm{mot}^5(N,{\ell_0} N')} \\
	{\mathcal{I}(G(\mathbf{Q}_{\ell_0})/G(\mathbf{Z}_{\ell_0}) )} \&\& {H_\mathrm{mot}^5(N,N')}
	\arrow["{\mathrm{Symbl}^{[a,b,d,r],U}_{\ell_0}}", from=1-1, to=1-3]
	\arrow["{\mathrm{Tr}^U_{G(\mathbf{Z}_{\ell_0})}}"', from=1-1, to=2-1]
	\arrow["{(\mathrm{pr}_1)_*}", from=1-3, to=2-3]
	\arrow["{\mathrm{Symbl}^{[a,b,d,r],\mathrm{sph}}_{\ell_0}}"', from=2-1, to=2-3]
\end{tikzcd}\]
where $U:=\mathrm{GSp}_4(\mathbf{Z}_{\ell_0})\times K_1({\ell_0})$ since $v_{\ell_0}(\ell_0NN')=1$. Moreover, the map $\mathrm{Symbl}^{[a,b,d,r],\mathrm{sph}}_{\ell_0}$ intertwines the action of the Hecke operator $$\ch(G(\mathbf{Z}_{\ell_0})(t_1,t_2) G(\mathbf{Z}_{\ell_0})),\ \ (t_1,t_2)\in T(\mathbf{Q}_{\ell_0})$$ 
on the source, with the action of $$\ch(G(\mathbf{Z}_{\ell_0})(t_1,t_2) G(\mathbf{Z}_{\ell_0}))\mathrm{Frob}_{\ell_0}^{v_{\ell_0}(\mu(t_1))}$$
on the target. Given an element $\mathcal{P}\in\mathcal{H}_{G(\mathbf{Q}_{\ell_0})}^\circ$ we write $\mathcal{P}^{\mathrm{Frob}_{\ell_0}}$ for this intertwined action on the target. Since ${\ell_0}|{\ell_0} N'$ and ${\ell_0}\nmid N$, we have by construction and the choice of input data t $R:=N\ell_0N'$ in \Cref{def input at R}, that $$_c\Xi_{\mathrm{mot}}(\mu[N,{\ell_0} N'])=\mathrm{Symbl}^{[a,b,d,r],U}_{\ell_0}\left(\delta_{{\ell_0},\mu}\right),\ \ \delta_{{\ell_0},\mu}:=C_{\ell_0}\phi_{\ell_0} \otimes\left\{\ch(\GSp_4(\mathbf{Z}_{\ell_0})-\ch(\eta_{\ell_0} \GSp_4(\mathbf{Z}_{\ell_0}))\otimes\xi_{{\ell_0},\mu}^{\GL_2}\right\}.$$
Moreover, since ${\ell_0}\nmid NN'$, it follows by the earlier choices made in \Cref{sec input for G}, that 
$$_c\Xi_{\mathrm{mot}}(\mu[N,N'])=\mathrm{Symbl}^{[a,b,d,r],\mathrm{sph}}_{\ell_0}\left(\delta_{{\ell_0},0}\right),\ \ \delta_{{\ell_0},0}:=\ch(\mathbf{Z}_{\ell_0}^2)\otimes\ch(G(\mathbf{Z}_{\ell_0})).$$
Thus, by the diagram above and \Cref{def unique operator}, we have 
\begin{align}\label{eq: 5}
    (\mathrm{pr}_1)_*\left( _c\Xi_{\mathrm{mot}}(\mu[N,{\ell_0} N'])\right)&=\mathrm{Symbl}^{[a,b,d,r],\mathrm{sph}}_{\ell_0}(\mathrm{Tr}(\delta_{{\ell_0},\mu})),\ \ \mathrm{Tr}({\delta}_{\ell_0,\mu}):=\mathrm{Tr}^U_{G(\mathbf{Z}_{\ell_0})}(\delta_{{\ell_0},\mu})\\
 \nonumber   &=\mathcal{P}_{\mathrm{Tr}(\delta_{{\ell_0},\mu})}^{\mathrm{Frob}_{\ell_0}}\cdot\ _c\Xi_{\mathrm{mot}}([N, N']).
\end{align}
It follows from \eqref{eq: 3} that $({\mathcal{P}_{\ell_0}^{\mathrm{spin}}}')^{\mathrm{Frob}_{\ell_0}}={\mathcal{P}_{\ell_0}^{\mathrm{spin}}}'(\mathrm{Frob}_{\ell_0}^{-1})$. Thus, it suffices to show that $\mathcal{P}_{\mathrm{Tr}(\delta_{{\ell_0},\mu})}={\mathcal{P}_{\ell_0}^{\mathrm{spin}}}'(\nu({\ell_0}))$ in the quotient $\mathcal{H}_{G(\mathbf{Q}_{\ell_0})}^\circ/\mathrm{ker(\Theta_{\pi^\vee})}.$ As in the proof of \cite[Theorem $5.2.4$]{Loeffler_2021}, we will do this by comparing the image of the period $\mathcal{Z}$ evaluated on $\mathrm{Tr}(\delta_{{\ell_0},\mu})$ and $\delta_{{\ell_0},0}$, for a sufficiently dense family of representations. 

Let $\Pi$ be irreducible unramified principal-series representation of $\GSp_4(\mathbf{Q}_{\ell_0})$. We consider the period 
$$\mathcal{Z}\in \mathrm{Hom}_{H(\mathbf{Q}_{\ell_0})}(\mathcal{S}(\mathbf{Q}_{\ell_0}^2)\otimes\Pi\otimes\pi,\mathbf{1}).$$This period lifts to a linear form on $\mathcal{I}(G(\mathbf{Q}_{\ell_0})/G(\mathbf{Z}_{\ell_0}))$ in the usual manner; i.e. 
$$\mathcal{I}(G(\mathbf{Q}_{\ell_0})/G(\mathbf{Z}_{\ell_0}))\ni \phi\otimes\xi\mapsto\mathcal{Z}(\phi\otimes \xi\cdot(W_\Pi^\mathrm{sph}\otimes W_\pi^\mathrm{sph})).$$
Thus, using the definition of the Hecke operator $\mathcal{P}_{\mathrm{Tr}(\delta_{{\ell_0},\mu})}$ and \Cref{lem linear form}, we have
\begin{align}\label{eq: 6}
\mathcal{Z}\left(C_{\ell_0}\phi_{\ell_0}\otimes(W_\Pi^\mathrm{sph}-\eta W_\Pi^\mathrm{sph})\otimes\left(\mathrm{Tr}^{K_1({\ell_0})}_{\GL_2(\mathbf{Z}_{\ell_0})}(\xi_{{\ell_0},\mu}^{\GL_2})\cdot W_{\pi}^\mathrm{sph}\right)\right)=\Theta_{\Pi\times \pi}\left(\mathcal{P}_{\mathrm{Tr}(\delta_{{\ell_0},\mu})}'\right).
\end{align}
For ease of notation, write $W_\Pi:=W_\Pi^\mathrm{sph}-\eta W_\Pi^\mathrm{sph}$ and $W_\pi:=\left(\mathrm{Tr}^{K_1({\ell_0})}_{\GL_2(\mathbf{Z}_{\ell_0})}(\xi_{{\ell_0},\mu}^{\GL_2})\right)\cdot W_\pi^\mathrm{sph}$.
The lefthand side of \eqref{eq: 6} is given by 
\begin{align}\label{eq: 7}\lim_{s\rightarrow 0}\frac{Z(C_{\ell_0}\phi_{\ell_0},W_\Pi,W_{\pi};s)}{L(\Pi\times\pi,s)}.\end{align}
By the definition of $\xi_{{\ell_0},\mu}^{\GL_2}$ and the fact that $\pi=I(\nu,\mu)$, one can check, for example using \cite[\S $6.2$]{loeffler2021zetaintegralsunramifiedrepresentationsgsp4}, that $W_\pi$ is precisely given by the Whittaker function $W^{\phi_{\ell_0}}(\nu,\mu)$ of \Cref{sec: siegel sections}, in the Whittaker model of $I(\nu,\mu)$. Thus, by the basic formula \eqref{eq: basic formula} we have that $Z(C_{\ell_0}\phi_{\ell_0},W_\Pi,W_{\pi};s)$ is given by the spin $L$-factor $L(\Pi\times\mu,s)$ times the Piatetski-Shapiro zeta integral $Z(C_{\ell_0}\tilde{B}_{W_\Pi},\phi_{\ell_0},\phi_{\ell_0};\Lambda,\nu,s)$. This coincides with the integral $Z(C_{\ell_0}\tilde{B}_{W_{\Pi\times\nu}},\phi_{\ell_0},\phi_{\ell_0};\Lambda\times\nu,1,s)$. Combining this with \eqref{eq: 7} and the fact that $L(\Pi\times\pi,s)=L(\Pi\times\mu,s)L(\Pi\times\nu,s)$, we see that the lefthand side of \eqref{eq: 6} is actually given by 
\begin{align}\label{eq: 8}
    \lim_{s\rightarrow 0}\frac{Z(C_{\ell_0}\tilde{B}_{W_{\Pi\times\nu}},\phi_{\ell_0},\phi_{\ell_0};\Lambda\times\nu,1,s)}{L(\Pi\times\nu,s)}.
\end{align}
We now claim that $Z(C_{\ell_0}\tilde{B}_{W_{\Pi\times\nu}},\phi_{\ell_0},\phi_{\ell_0};\Lambda\times\nu,1,s)=1$. For ease of notation, we write $\sigma:=\Pi\times\nu$ and $\lambda=(\lambda_1,\lambda_2):=\Lambda\times\nu$. Applying the same unfolding of the integral as the one in \cite[Proposition $2.1$]{rösner2020spinoreulerfactorsgsp4} (after translating to our matrix model and notation), we see that $Z(\tilde{B}_{W_{\Pi\times\mu}},\phi_{\ell_0},\phi_{\ell_0};\Lambda\times\mu,1,s)$ is given by 
\begin{align}\label{eq: 9}
    \int_{\mathrm{SL}_2(\mathbf{Z}_{\ell_0})\times\mathrm{SL}_2(\mathbf{Z}_{\ell_0})}\int_{\mathbf{Q}_{\ell_0}^\times\times\mathbf{Q}_{\ell_0}^\times}Z_{\mathrm{reg}}^{\mathrm{PS}}((k_1,k_2)\tilde{B}_{W_\sigma};s)\phi_{\ell_0}((0,x)k_1)\phi_{\ell_0}((0,y)k_2)\lambda_1(x)\lambda_2(y)|xy|^{s+\tfrac{1}{2}}\ dk_1dk_2d^\times x d^\times y.
\end{align}
where $dk_i$ is normalized to give $\mathrm{SL}_2(\mathbf{Z}_{\ell_0})$ volume $1$ and 
$$Z_\mathrm{reg}^\mathrm{PS}((k_1,k_2)\tilde{B}_{W_\sigma};s)=\int_{\mathbf{Q}_{\ell_0}^\times} \tilde{B}_{W_\sigma}(\left[\begin{smallmatrix}
    a & & & \\
    & a & & \\
    & & 1& \\
    & & & 1
\end{smallmatrix}\right](k_1,k_2))|a|^{s-\tfrac{3}{2}}\ d^\times a.$$
Let 
$$U_0:=\mathrm{SL}_2(\mathbf{Z}_{\ell_0})\cap \mathrm{Iw}(\ell_0^2),\ \ U_1:=\mathrm{SL}_2(\mathbf{Z}_{\ell_0})\cap K_1(\ell_0^2),\ \ K_1(\ell_0^2):=\{k\in\GL_2(\mathbf{Z}_{\ell_0})\ |\ k\equiv \left[\begin{smallmatrix}
    * & * \\
    0& 1
\end{smallmatrix}\right]\mod \ell_0^2\}.$$
One checks that $\eta^{-1} (U_1\times_{\GL_1} U_1)\eta\subseteq\GSp_4(\mathbf{Z}_{\ell_0})$ and thus $U_1\times_{\GL_1} U_1$ stabilizes $\tilde{B}_{W_\sigma}$ by the definition of $W_\sigma.$ Since it also stabilizes $\phi_{\ell_0}\times\phi_{\ell_0}$, the quadruple integral in \eqref{eq: 9} is given by 
\begin{align*}
    \sum_{\substack{k_1,k_2\ \mathrm{in}\\\mathrm{SL}_2(\mathbf{Z}_{\ell_0})/U_1(\ell_0^2)}}V^2\cdot Z_\mathrm{reg}^\mathrm{PS}((k_1,k_2)\tilde{B}_{W_\sigma};s)\int_{\mathbf{Q}_{\ell_0}^\times}\phi_{\ell_0}((0,x)k_1)\lambda_1(x)|x|^{s+\tfrac{1}{2}}\ d^\times x \int_{\mathbf{Q}_{\ell_0}^\times}\phi_{\ell_0}((0,y)k_2)\lambda_2(y)|y|^{s+\tfrac{1}{2}}\ d^\times y
\end{align*}
where $V:=\vol(U_1(\ell_0^2),dk_i)=\ell_0^{-2}(\ell_0^2-1)^{-1}.$ Since the characters $\lambda_i$ are unramified by construction, the two Siegel sections associated to $\phi_{\ell_0}$ as functions of $k_i$, are supported on $U_0(\ell_0^2)$ where they take the value ${\ell_0}^{-1}({\ell_0}-1)^{-1}.$ Thus, \eqref{eq: 9} becomes 
$$\frac{V^2}{\ell_0^2({\ell_0}-1)^2}\cdot \sum_{\substack{k_1,k_2\ \mathrm{in}\\U_0(\ell_0^2)/U_1(\ell_0^2)}} Z_\mathrm{reg}^\mathrm{PS}((k_1,k_2)\tilde{B}_{W_\sigma};s).$$
A complete set of distinct coset representatives of $U_0(\ell_0^2)/U_1(\ell_0^2)$ is given by $\{m_u:=\left[\begin{smallmatrix}
    u^{-1} & \\
    & u 
\end{smallmatrix}\right]\ |\ u\in(\mathbf{Z}/\ell_0^2\mathbf{Z})^\times\}.$ Computing $(m_u,m_v)\cdot\eta \cdot(m_u,m_v)^{-1}$ and using the equivariance properties of the Bessel funtion, the integral $Z_\mathrm{reg}^\mathrm{PS}((m_u,m_v)\tilde{B}_{W_\sigma};s)$ is given by $\int_{\mathbf{Q}_{\ell_0}^\times}\{1-\psi(a(uv{\ell_0})^{-1})\}\tilde{B}_{W_\sigma^\mathrm{sph}}(\mathrm{diag}(a,a,1,1))|a|^{s-\tfrac{3}{2}}\ d^\times a$, which is simply equal to $\frac{{\ell_0}}{{\ell_0}-1}$. Hence \eqref{eq: 9} simplifies to $\frac{{\ell_0} V^2}{{\ell_0}-1}=\frac{1}{C_{\ell_0}}$ as claimed. Thus, from \Cref{eq: 6} and \eqref{eq: 8}, we obtain 
\begin{align}
    \Theta_{\Pi\times\pi}\left( \mathcal{P}_{\mathrm{Tr}(\delta_{{\ell_0},\mu)})}'\right)=L(\Pi\times\nu,0)^{-1}=\Theta_{\Pi\times\pi}\left(\mathcal{P}_{\ell_0}^{\mathrm{spin}}(\nu({\ell_0}))\right).
\end{align}
Then it immediately follows that 
\begin{align*}
    \Theta_{\Pi^\vee\times\pi^\vee}\left( \mathcal{P}_{\mathrm{Tr}(\delta_{{\ell_0},\mu)})}\right)=\Theta_{\Pi^\vee\times\pi^\vee}\left(\mathcal{P}_{\ell_0}^{\mathrm{spin}'}(\nu({\ell_0}))\right).
\end{align*}
This holds for all such representations $\Pi$ and thus the result then follows by a density argument in the Hecke algebra.
\end{proof}

\begin{prop}\label{prop: integrality} Suppose $p$ is now coprime to $MNN'$. Then the image of the class $_c\Xi_{\mathrm{mot}}(\mu[N,N'])$ under the $p$-adic \'etale regulator map lies in the cohomology with $\overline{\mathbf{Z}}_p$-coefficients.
    \begin{proof}
        We need to check integrality prime by prime. For primes dividing $pN$ this follows by definition of the input data. For primes not dividing $pNN'$, this follows by part of \cite[Theorem $8.7$] {hsu2020eulersystemsmathrmgsp4times} (the non-trivial cases being the primes dividing $M$, which we have fixed here). It remains to check primes dividing $N'$ and not $N$. In other words, if $\ell$ is such a prime, we need to check that the local input data $C_\ell\phi_{\ell,2}\otimes\xi_{\ell,\mu}$ of \Cref{def input at R}, is $\overline{\mathbf{Z}}_p$-integral at level $\GSp_4(\mathbf{Z}_\ell)\times K_1(\ell^e)$, in the sense of \cite[\S $3$]{Loeffler_2021}. Since $\ell\neq p$ we may assume that $e=1$ since powers of $\ell$ are inconsequential. Write 
        \begin{align*}K_1(\ell^2)&:=\{k\in\GL_2(\mathbf{Z}_\ell)\ |\ k\equiv\left[\begin{smallmatrix}
            * & * \\
            0 & 1
        \end{smallmatrix}\right]\mod \ell^2\}\\
        K^1(\ell^2)&:=\{k\in\GL_2(\mathbf{Z}_\ell)\ |\ k\equiv\left[\begin{smallmatrix}
            1 & * \\
            0 & *
        \end{smallmatrix}\right]\mod \ell^2\}
        \end{align*}
        Let $m\in\GL_2(\mathbf{Q}_\ell)$ be either of the matrices $\left[\begin{smallmatrix}
            & -1\\
            \ell^2 & 
        \end{smallmatrix}\right],\left[\begin{smallmatrix}
            & -\ell\\
            \ell^2 & 
        \end{smallmatrix}\right]$, which appear in the definition of $\xi_{\ell,\mu}^{\GL_2}$ in \Cref{def input at R}. An eassy check shows that  $$K_1(\ell^2)\times_{\GL_1} K^1(\ell^2) \subseteq\mathrm{Stab}_H(\phi_{\ell,2})\cap(\eta_\ell,m)\cdot\left(\GSp_4(\mathbf{Z}_\ell)\times K_1(\ell)\right)\cdot (\eta_\ell,m)^{-1}.$$
        Moreover, $\vol_{H(\mathbf{Q}_\ell)}(K_1(\ell^2)\times_{\GL_1} K^1(\ell^2))^{-1}=\ell^4(\ell^2-1)^2$ which coincides with $\frac{C_\ell}{\ell-1}$ up to a power of $\ell\neq p$. This concludes the proof.  
    \end{proof}
\end{prop}

\section{Norm-relations over ray class groups}\label{sec norm-relations over ray class groups}
\subsection{Introducting a Groessencharacter}\label{sec introducing a grossernacharacter}
We follow the setup of \cite{Lei_Loeffler_Zerbes_2015}. Let $K/\mathbf{Q}$ be an imaginary quadratic field of square-free discriminant $\Delta_K$. Let $\psi$ a Groessencharacter of $K$ of infinity-type $(-1,0)$ and modulus $\mathfrak{f}$, taking values in a Galois number field $L$. We let $\omega$ be the unique Dirichlet character modulo $N_{K/\mathbf{Q}}(\mathfrak{f})$ for which $\psi((n))=n\omega(n)$ for all $(n,N_{K/\mathbf{Q}}(\mathfrak{f}))=1$.  Henceforth, we also assume that the integer $M$ defining the cyclotomic field which was fixed in the previous section is $1.$
\begin{thm}\emph{(}\cite[Theorem $3.1.1$]{Lei_Loeffler_Zerbes_2015}, \cite[Theorem $4.8.2$]{miyake2006modular}\emph{)}\label{thm miyaki}
    The formal $q$-expansion 
    $$\sum_{\substack{\mathfrak{a}\subseteq\mathcal{O}_K\\(\mathfrak{a},\mathfrak{f})=1}}\psi(\mathfrak{a})q^{N_{K/\mathbf{Q}}(\mathfrak{a})},$$
    is the $q$-expansion of a normalized cuspidal Hecke eigenform 
    $g_\psi\in S_2(\Gamma_1(N_\mathfrak{f}),\omega\varepsilon_K)$,
    where $N_\mathfrak{f}:=N_{K/\mathbf{Q}}(\mathfrak{f})\Delta_K$ and $\varepsilon_K$ is the quadratic Dirichlet character associated to $K$.
\end{thm}
Let $\mathfrak{n\subseteq\mathcal{O}}_K$ be an integral ideal with $\mathfrak{f}|\mathfrak{n}$, and let $N_\mathfrak{n}:=N_{K/\mathbf{Q}}(\mathfrak{n})\Delta_K$ which is a multiple of $N_\mathfrak{f}$. Let $H_\mathfrak{n}$ be the ray class group of $K$ of modulus $\mathfrak{n}$. Let $\mathbf{T}^{\GL_2}_{L,N_\mathfrak{n}}$ be the commutative $L$-subalgebra of $\mathbf{C}[U_1(N_\mathfrak{n})\backslash\GL_2(\mathbf{A}_\mathrm{f})/U_1(N_\mathfrak{n})]$ given by the polynomial algebra $L[T_{\mathfrak{n},\ell}':\ell\ \mathrm{prime}][S_{\ell}'^{\pm 1}:\ell\nmid N_\mathfrak{n}\ \mathrm{prime}]$ where 
\begin{align*}
    T_{\mathfrak{n},\ell}'=\begin{dcases}
        \ch(K_1(\ell^e)\left[\begin{smallmatrix}
            \ell^{-1} & \\
            & 1
        \end{smallmatrix}\right] K_1(\ell^e)),\ &\mathrm{if}\ \ell^e\parallel N_\mathfrak{n}\\
        \ch(\GL_2(\mathbf{Z}_\ell)\left[\begin{smallmatrix}
            \ell^{-1} & \\
            & 1
        \end{smallmatrix}\right]\GL_2(\mathbf{Z}_\ell)),\ &\mathrm{if}\ \ell\nmid N_\mathfrak{n}
    \end{dcases},\ \ \ 
    S_{\ell}'=
        \ch(\left[\begin{smallmatrix}
            \ell^{-1} & \\
            & \ell^{-1}
\end{smallmatrix}\right]\GL_2(\mathbf{Z}_\ell)). 
\end{align*} 
We will consider the following algebra homomorphism, as in \cite[\S $3.2$]{Lei_Loeffler_Zerbes_2015}. For $R\in\{L,\mathcal{O}_L\}$, we set
\begin{align*}
    \phi_\mathfrak{n}: \mathbf{T}_{R,N_\mathfrak{n}}^{\GL_2}  &\longrightarrow R[H_\mathfrak{n}]\\
    T_{\mathfrak{n},\ell}'&\mapsto \sum_{\substack{\mathfrak{l}\nmid\mathfrak{n}\\ N_{K/\mathbf{Q}}(\mathfrak{l})=\ell}}[\mathfrak{l}]{\psi(\mathfrak{l})}\\
    S_{\ell}'&\mapsto [(\ell)]{(\omega\varepsilon_K)(\ell)}.
\end{align*}
\noindent It has the property that for any character $\chi$ of $H_\mathfrak{n}$, the evaluation at $\chi$ map composed with $\phi_\mathfrak{n}$, corresponds to the Hecke eigensystem of the eigenform $g_{\psi\chi}$, i.e. the form attached to the twist of $\psi$ by $\chi$ through \Cref{thm miyaki}.
\subsection{Introducing a $\GSp_4$ automorphic representation and its associated Galois representation}\label{sec  GSp4 aut rep} Let $\Pi=\Pi_\infty\otimes\Pi_\mathrm{f}$ be a cuspidal automorphic representation of $\GSp_4$ with $\Pi_\infty$ cohomological unitary discrete series representation of weight $(k_1,k_2)$, $k_1\geq k_2\geq 3$. We set $w:=k_1+k_2-3$, and write $\omega_\Pi$ for the finite-order Hecke character given by the central character of $\Pi.$
\begin{thm}\emph{(\cite[Theorem I]{weissauer2005four})}\label{thm galois rep}
    Let $\Pi$ be as above and further assume that $\Pi$ is of general-type. Let $S$ be a finite set of primes containing all the primes at which $\Pi$ ramifies. Then the following hold:
    \begin{enumerate}
        \item For any prime $\ell\notin S$, the local component $\Pi_\ell$ is an unramified principal-series representation.
        \item Let $\mathcal{P}_\ell^\mathrm{spin}(X)\in\mathcal{H}_{\GSp_4(\mathbf{Q}_\ell)}^\circ[X]$ be the polynomial interpolating local spin $L$-factors as in the previous section. The subfield $L\subseteq\mathbf{C}$ generated by the coefficients of $\Theta_{\Pi_\ell}(\mathcal{P}_\ell^\mathrm{spin})(\ell^{\tfrac{w}{2}}X)$ for all $\ell\notin S$, is a finite extension of $\mathbf{Q}.$
        \item For any prime $p$ and any choice of embedding $L\hookrightarrow\overline{\mathbf{Q}}_p$, there is a semi-simple four-simensional Galois representation $$\rho_{\Pi,p}:\mathrm{Gal}(\overline{\mathbf{Q}}/\mathbf{Q})\rightarrow \GL_4(\overline{\mathbf{Q}}_p)$$ characterised up to isomorphism by the property that for all primes $\ell\notin S\cup\{p\}$, we have
        $$P_\ell(X):=\det\left(1-X\rho_{\Pi,p}(\mathrm{Frob}_\ell^{-1})\right)= \Theta_{\Pi_\ell}(\mathcal{P}_\ell^\mathrm{spin})(\ell^{\tfrac{w}{2}}X).$$
        \item For all isomorphisms $\overline{\mathbf{Q}}_p\simeq\mathbf{C}$ the image of the eigenvalues of $\rho_{\Pi,p}(\mathrm{Frob}_\ell^{-1})$ has absolute value $\ell^{w/2}$ for all $\ell\not\in S\cup\{p\}.$
    \end{enumerate}
\end{thm}
Let $N_\Pi:=\prod_{\ell\in S}\ell$ and write  $\mathcal{H}_{L,N_\Pi}^{\GSp_4}$ for the $L$-valued spherical Hecke algebra of $\GSp_4$ away from $N_\Pi$. Let
$\Theta_{\Pi^\vee}[\tfrac{w-3}{2}]$ be the spherical Hecke eigensystem of $\Pi^\vee\otimes|\mu(-)|_\mathbf{A}^{(w-3)/2}$ where $\mu$ here denotes the multiplier map. Since all the local spherical Hecke algebras (away from $N_\Pi$) are generated by the Hecke operators $\mathscr{T}_\ell,\mathscr{R}_\ell,\mathscr{S}_\ell$ of \eqref{eq: Hecke ops} (for example \cite[\S $5.1.3$]{pilloni2020higher}), it follows from \Cref{thm galois rep} parts $(2)$ and $(3)$ and the explicit expression of the spin $L$-factors given earlier, that the spherical Hecke eigensystem gives a morphism 
$\Theta_{\Pi^\vee}[\tfrac{w-3}{2}]:\mathcal{H}_{L,N_\Pi}^{\GSp_4}\longrightarrow L.$
\subsection{Quotients of cohomology, norm maps and classes}\label{sec quotients of coh.} We will now slightly change the cohomology groups used. We set 
$$\mathbf{H}_\mathrm{mot}^5(N_\Pi,N_\mathfrak{n}):= H^5_\mathrm{mot}(Y_{\GSp_4}(U^{\GSp_4}(N_\Pi))\times Y_{\GL_2}(U^{\GL_2}(N_\mathfrak{n})),\mathscr{W}_\mathbf{Q}^{a,b,d,*}(3-a-r)).$$
In particular, we get rid of the twist $[-a-r]$, so that $H_\mathrm{mot}^5(N_\Pi,N_\mathfrak{n})=\mathbf{H}_\mathrm{mot}^5(N_\Pi,N_\mathfrak{n})[-a-r].$ The algebra $\mathcal{H}_{L,N_\Pi}^{\GSp_4}\otimes_L \mathbf{T}_{L,N_\mathfrak{n}}^{\GL_2}$ naturally acts on $\mathbf{H}_\mathrm{mot}^5(N_\Pi,N_\mathfrak{n})$.

\begin{defn}\label{def quot of H5}
    We define $\mathbf{H}_\mathrm{mot}^5(\Pi,\psi,\mathfrak{n},L)$ to be 
$$L[H_\mathfrak{n}]\underset{\left(\Theta_{\Pi^\vee}[\tfrac{w-3}{2}]\otimes\phi_\mathfrak{n},\mathcal{H}_{L,N_\Pi}^{\GSp_4}\otimes_L\mathbf{T}_{L,N_\mathfrak{n}}^{\GL_2}\right)}{\bigotimes}\mathbf{H}_\mathrm{mot}^5(N_\Pi,N_\mathfrak{n})_L.$$
    Let $E/L$ be a finite field extension. For any $E$-valued character $\chi$ of $H_\mathfrak{n}$, we write $\mathbf{H}_\mathrm{mot}^5(\Pi,\psi,\mathfrak{n},L)_{E,\chi}$ for the largest quotient of $$\mathbf{H}_\mathrm{mot}^5(\Pi,\psi,\mathfrak{n},L)\otimes_L E=L[H_\mathfrak{n}]\underset{\left(\Theta_{\Pi^\vee}[\tfrac{w-3}{2}]\otimes\phi_\mathfrak{n},\mathcal{H}_{L,N_\Pi}^{\GSp_4}\otimes_L\mathbf{T}_{L,N_\mathfrak{n}}^{\GL_2}\right)}{\bigotimes}\mathbf{H}_\mathrm{mot}^5(N_\Pi,N_\mathfrak{n})_E$$on which $H_\mathfrak{n}$ acts through $\chi.$ 
\end{defn}
\begin{defn}\label{def class Xi'}
   Let $c\in\mathbf{Z}_{>0}$ be coprime to $6N_\Pi N_\mathfrak{f}$. 
   \begin{enumerate}
  \item  We firstly define the class $_c\Xi_\mathrm{mot}'([N_\Pi,N_{\mathfrak{n}}])\in \mathbf{H}_\mathrm{mot}^5(N_\Pi,N_\mathfrak{n})$ to be the class constructed as in \emph{\Cref{def classes Xi_mot}}, but with $\mu=\mathbf{1}$ and $\xi_{\ell,\mu}^{\GL_2}$ of \emph{\Cref{def input at R}} replaced with $${\xi_{\ell,\mathbf{1}}^{\GL_2}}':=\ell^{-1}(\ell-1)^{-1}\ch(\left[\begin{smallmatrix}
       & -1 \\
       \ell^2 & 
   \end{smallmatrix}\right]K_\ell^{\GL_2})$$
   for all primes $\ell.$
   \item  We now define the class $_c\Xi_\mathfrak{n}^{\Pi,\psi}:=1\otimes 
 \ _c\Xi_\mathrm{mot}'([N_\Pi,N_{\mathfrak{n}}])$  in $\mathbf{H}_\mathrm{mot}^5(\Pi,\psi,\mathfrak{n},L)$
   \end{enumerate}
   
\end{defn}
Recall the natural projection map $\mathrm{pr_1}: Y_{\GL_2}(U^{\GL_2}(N_\mathfrak{n}\ell))\rightarrow Y_{\GL_2}(U^{\GL_2}(N_\mathfrak{n}))$ for primes $\ell\nmid N_\mathfrak{n}$. We will introduce two more degeneracy maps which we call $\mathrm{pr}_2$ and $\mathrm{pr}_2'$.
\begin{defn} Let $\mathfrak{n}$ be as above and $\ell\nmid N_\mathfrak{n}$.
\begin{enumerate}
\item We write $$\mathrm{pr}_2:Y_{\GL_2}(U^{\GL_2}(N_\mathfrak{n}\ell))\rightarrow Y_{\GL_2}(\left[\begin{smallmatrix}
    \ell & \\
    & 1
\end{smallmatrix}\right]U^{\GL_2}(N_\mathfrak{n}\ell)\left[\begin{smallmatrix}
    \ell^{-1} & \\
    & 1
\end{smallmatrix}\right])\rightarrow Y_{\GL_2}(U^{\GL_2}(N_\mathfrak{n}))$$
where the first arrow is given by right multiplication by $\left[\begin{smallmatrix}
    \ell^{-1} & \\
    & 1
\end{smallmatrix}\right]\in \GL_2(\mathbf{Q}_\ell)$ and the second arrow corresponds to the inclusion $\left[\begin{smallmatrix}
    \ell & \\
    & 1
\end{smallmatrix}\right]U^{\GL_2}(N_\mathfrak{n}\ell)\left[\begin{smallmatrix}
    \ell^{-1} & \\
    & 1
\end{smallmatrix}\right]\subseteq U^{\GL_2}(N_\mathfrak{n}).$
    \item We write $$\mathrm{pr}_2':Y_{\GL_2}(U^{\GL_2}(N_\mathfrak{n}\ell))\rightarrow Y_{\GL_2}(\left[\begin{smallmatrix}
    1 & \\
    & \ell^{-1}
\end{smallmatrix}\right]U^{\GL_2}(N_\mathfrak{n}\ell)\left[\begin{smallmatrix}
    1 & \\
    & \ell
\end{smallmatrix}\right])\rightarrow Y_{\GL_2}(U^{\GL_2}(N_\mathfrak{n}))$$
where the first arrow is given by right multiplication by $\left[\begin{smallmatrix}
    1 & \\
    & \ell
\end{smallmatrix}\right]\in\GL_2(\mathbf{Q}_\ell)$, and the second arrow corresponds to the inclusion $\left[\begin{smallmatrix}
    1 & \\
    & \ell
\end{smallmatrix}\right]^{-1}U^{\GL_2}(N_\mathfrak{n}\ell)\left[\begin{smallmatrix}
    1 & \\
    & \ell
\end{smallmatrix}\right]\subseteq U^{\GL_2}(N_\mathfrak{n}).$ 
\end{enumerate}
As usual, we denote by $(\mathrm{pr}_2)_*$ and $(\mathrm{pr}_2')_*$ the corresponding pushforward maps $\mathbf{H}_\mathrm{mot}^5(N_\Pi,N_\mathfrak{n}\ell)\rightarrow \mathbf{H}_\mathrm{mot}^5(N_\Pi,N_\mathfrak{n})$.
\end{defn}
It follows by construction that $(\mathrm{pr}_2)_*$ and $(\mathrm{pr}_2')_*$ are related via 
\[\begin{tikzcd}[ampersand replacement=\&]
	{\mathbf{H}_\mathrm{mot}^5(N_\Pi,N_\mathfrak{n}\ell)} \& {\mathbf{H}_\mathrm{mot}^5(N_\Pi,N_\mathfrak{n})} \\
	\& {\mathbf{H}_\mathrm{mot}^5(N_\Pi,N_\mathfrak{n})}
	\arrow["{(\mathrm{pr}_2)_*}", from=1-1, to=1-2]
	\arrow["{(\mathrm{pr}_2')_*}"', from=1-1, to=2-2]
	\arrow["{\cdot S_{\ell}=\mathrm{ch}(\mathrm{diag}(\ell,\ell)\mathrm{GL}_2(\mathbf{Z}_\ell))}"', from=2-2, to=1-2]
\end{tikzcd}\]
\begin{defn}\emph{(\cite[Definition $3.3.1$]{Lei_Loeffler_Zerbes_2015})}\label{def set A and norm map} Let $c\in\mathbf{Z}_{>0}$ be an integer coprime to $6N_\Pi N_\mathfrak{f}$. Write $\mathcal{A}_c$ for the set of integral ideals of $K$ divisible by $\mathfrak{f}$, whose norm is coprime to $c$.
     Let $\mathfrak{n},\mathfrak{nl}\in\mathcal{A}_c$. 
     \begin{enumerate}
     \item If $\ell=\mathfrak{l}\overline{\mathfrak{l}}\nmid N_\mathfrak{n}$ is a split prime. We define formally 
$$\mathcal{N}^{\mathfrak{nl}}_\mathfrak{n}:=[\overline{\mathfrak{l}}]^{-2}\psi(\overline{\mathfrak{l}})^{-2}\left\{1\otimes(\mathrm{pr}_1)_*-\frac{[\mathfrak{l}]\psi(\mathfrak{l})}{\ell}\otimes(\mathrm{pr}_2)_*\right\}.$$
\item If $\mathfrak{l}|\mathfrak{n}$, we define formally 
$$\mathcal{N}^{\mathfrak{nl}}_\mathfrak{n}:=1\otimes (\mathrm{pr}_1)_*.$$
\end{enumerate}
\end{defn}

\begin{rem}Notice the extra normalization factor of $[\overline{\mathfrak{l}}]^{-2}\psi(\overline{\mathfrak{l}})^{-2}$ in our definition of $\mathcal{N}_\mathfrak{n}^{\mathfrak{nl}}$ for split primes, which is not there in \textit{loc.cit}. Here it is convenient to include it, and it's purpose will become apparent soon enough. 
\end{rem}
\begin{lem}
    For $\mathfrak{n},\mathfrak{nl}\in\mathcal{A}_c$ as above. The map $\mathcal{N}^{\mathfrak{ln}}_\mathfrak{n}$ gives rise to a well-defined morphism 
    $$\mathcal{N}^{\mathfrak{nl}}_\mathfrak{n}:\mathbf{H}_\mathrm{mot}^5(\Pi,\psi,\mathfrak{nl},L)\rightarrow \mathbf{H}_\mathrm{mot}^5(\Pi,\psi,\mathfrak{n},L).$$
    \begin{proof}
    This is essentially an adelic formulation of (part of) \cite[Proposition $3.3.2$]{Lei_Loeffler_Zerbes_2015} in our setup. However, this once, we will give the details for completeness. We do this for $\mathfrak{l}\nmid\mathfrak{n}$ since otherwise it is clear. Away from the prime $\mathfrak{l}$ this follows formally by construction. At the prime $\mathfrak{l}$, it suffices to show that $\mathcal{N}_\mathfrak{n}^\mathfrak{nl}(1\otimes (T_{\mathfrak{nl},\ell}'\cdot x))=\mathcal{N}_\mathfrak{n}^\mathfrak{nl}(\phi_\mathfrak{nl}(T_{\mathfrak{nl},\ell}')\otimes x)$ as elements of $\mathbf{H}_\mathrm{mot}^5(\Pi,\psi,\mathfrak{n},L)$, where $x$ denotes an arbitrary element of $\mathbf{H}_\mathrm{mot}^5(\Pi,\psi,\mathfrak{nl},L)$, and $\ell$ is the prime below $\mathfrak{l}$. For ease of notation, we will set $K:=\GL_2(\mathbf{Z}_\ell)$ and $t_\ell:=\left[\begin{smallmatrix}
            1 & \\
             & \ell
        \end{smallmatrix}\right]$ for the rest of the proof. We start by noting that since $\ell\parallel N_\mathfrak{nl}$, 
    \begin{align}\label{eq:11}
        (\mathrm{pr}_1)_*\cdot T_{\mathfrak{nl},\ell}'\cdot x=S_\ell'\cdot (\mathrm{pr}_1)_*\cdot T_{\mathfrak{nl},\ell}\cdot x&=S_\ell'\cdot\sum_{k_1\in K/K_1(\ell)}k_1\cdot \sum_{k_2\in K_1(\ell)t_\ell K_1(\ell)/K_1(\ell)}k_2\cdot x\\
   \nonumber     &=S_\ell'\cdot\sum_{k_1\in K/K_1(\ell)}k_1\cdot\sum_{k_2\in K_1(\ell)/K_1(\ell)\cap t_\ell K_1(\ell) t_\ell^{-1}} k_2\cdot t_\ell\cdot x\\
     \nonumber   &=S_\ell'\cdot\sum_{k\in K/K_1(\ell)\cap t_\ell K_1(\ell) t_\ell^{-1}}k\cdot t_\ell\cdot x.
    \end{align}
    On the other hand since $\ell\nmid N_\mathfrak{n}$, we have 
    \begin{align}\label{eq: 12}
    T_{\mathfrak{n},\ell}'\cdot(\mathrm{pr}_1)_*\cdot x=S_\ell'\cdot T_{\mathfrak{n},\ell}\cdot(\mathrm{pr}_1)_*\cdot x&=S_\ell'\cdot\sum_{k_1\in Kt_\ell K/K}k_1\cdot \sum_{k_2\in K/K_1(\ell)} k_2\cdot x\\
    \nonumber &=S_\ell'\cdot \sum_{k_1\in K/K\cap t_\ell K t_\ell^{-1}}k_1 \cdot t_\ell\cdot\sum_{k_2\in K/K_1(\ell)} k_2\cdot x\\
    \nonumber &=S_\ell'\cdot\sum_{k_1\in K/K\cap t_\ell K t_\ell^{-1}} k_1 \cdot \sum_{k_2\in t_\ell Kt_\ell^{-1}/t_\ell K_1(\ell) t_\ell^{-1}} k_2\cdot t_\ell \cdot x.
    \end{align}
    One easily checks that the natural map 
    $$\frac{t_\ell K t_\ell^{-1} \cap K}{t_\ell K_1(\ell) t_\ell^{-1}\cap K_1(\ell)}\hookrightarrow  \frac{t_\ell K t_\ell^{-1}}{t_\ell K_1(\ell) t_\ell^{-1}}$$
    is a strict injection and a complete set of distinct coset representatives of the righthand side which do not lie in the image of this map, is given by $\{t_\ell \left[\begin{smallmatrix}
        & u\\
        1 &
    \end{smallmatrix}\right]t_\ell^{-1}\ |\ u\in (\mathbf{Z}/\ell \mathbf{Z})^\times\}$.
    Combining this with \eqref{eq: 12}, \eqref{eq:11}, the fact that $K\cap t_\ell K t_\ell^{-1}=K_0(\ell)$, and fixing a complete list of distinct coset representatives $K/K_0(\ell)=\{
        \left[\begin{smallmatrix}
            1 & \\
            v & 1
        \end{smallmatrix}\right],\left[\begin{smallmatrix}
            & 1\\
            1 &
        \end{smallmatrix}\right]\ |\ v\in \mathbf{Z}/\ell \mathbf{Z}\}$, we have 
    \begin{align}\label{eq:13}
        T_{\mathfrak{n},\ell}'\cdot (\mathrm{pr}_1)_*\cdot x&=\left\{S_\ell'\cdot\sum_{k\in K/K_1(\ell)\cap t_\ell K_1(\ell) t_\ell^{-1}}k_1\cdot t_\ell \cdot x\right\}+\left\{S_\ell'\cdot\sum_{k_2\in K/K_0(\ell)}k_2\cdot \sum_{u\in (\mathbf{Z}/\ell\mathbf{Z})^\times} \left[\begin{smallmatrix}
            & u\\
            1 &
        \end{smallmatrix}\right]\left[\begin{smallmatrix}
            \ell & \\
            & 1
        \end{smallmatrix}\right]\cdot x\right\}\\
        \nonumber  &=\left\{(\mathrm{pr}_1)_*\cdot T_{\mathfrak{nl},\ell}'\cdot x\right\} + \left\{ S_{\ell}'\cdot (\mathrm{pr}_2)_*\cdot x\right\}
    \end{align}
    where the second equality follows from the fact that the $k_2\left[\begin{smallmatrix}
         & u\\
         1 & 
    \end{smallmatrix}\right]$ form a complete set of distinct coset representatives for $K/t_\ell^{-1} K_1(\ell) t_\ell$, and the definition of the map $\mathrm{pr_2}.$ Thus, all in all we have that $(\mathrm{pr}_1)_*\cdot T_{\mathfrak{nl},\ell}'\cdot x= \{T_{\mathfrak{n},\ell}'\cdot (\mathrm{pr}_1)_*\cdot x\}-\{S_{\ell}'\cdot (\mathrm{pr}_2)_*\cdot x\}.$ We now consider the second term in the definition of $\mathcal{N}^\mathfrak{nl}_\mathfrak{n}$. We have 
    \begin{align}\label{eq: 14}
        (\mathrm{pr}_2)_*\cdot T_{\mathfrak{nl},\ell}'\cdot x&=S_\ell'\cdot\sum_{k_1\in K/t_\ell^{-1} K_1(\ell)t_\ell}k_1\cdot \left[\begin{smallmatrix}
            \ell & \\
            & 1
        \end{smallmatrix}\right]\cdot \sum_{k_2\in K_1(\ell)t_\ell K_1(\ell)/ K_1(\ell) }k_2\cdot x\\
        \nonumber &=\sum_{k\in K/K_1(\ell)\cap t_\ell^{-1}K_1(\ell)t_\ell}k\cdot x\\
        \nonumber&=[K_1(\ell):K_1(\ell) \cap t_\ell^{-1}K_1(\ell)t_\ell]\sum_{k\in K/K_1(\ell)} k\cdot x\\
        \nonumber &=\ell(\mathrm{pr}_1)_*\cdot x
    \end{align}
    where the third equality follows from the fact that $x$ is already $K_1(\ell)$-invariant by definition, and the fourth equality follows from a simple calculation of the index and the definition of $\mathrm{pr}_1$. Putting \eqref{eq:13} and \eqref{eq: 14} together, while omitting the inconsequential factor of $[\overline{\mathfrak{l}}]^{-2}\psi(\overline{\mathfrak{l}})^{-2}$, we have 
    \begin{align}\label{eq: 15a}
\mathcal{N}^\mathfrak{nl}_\mathfrak{n}(1\otimes (T_{\mathfrak{nl},\ell}'\cdot x))&=\left(\phi_\mathfrak{n}(T_{\mathfrak{n},\ell}')-[{\mathfrak{l}}]\psi({\mathfrak{l}})\right)\otimes(\mathrm{pr}_1)_*\cdot x-\phi_\mathfrak{n}(S_{\ell}')\otimes(\mathrm{pr}_2)_*\cdot x
    \end{align}
    and we wish to verify that this is coincides with $\mathcal{N}^\mathfrak{nl}_\mathfrak{n}(\phi_\mathfrak{nl}(T_{\mathfrak{nl},\ell}')\otimes x)$. But $\phi_\mathfrak{n}(T_{\mathfrak{n},\ell}')=[\mathfrak{l}]\psi(\mathfrak{l})+[\overline{\mathfrak{l}}]\psi(\overline{\mathfrak{l}})$, $\phi_\mathfrak{n}(S_{\ell}')=[(\ell)](\nu \varepsilon_K)(\ell)=[\mathfrak{l}\overline{\mathfrak{l}}]\psi(\mathfrak{l}\overline{\mathfrak{l}})$ and $\phi_\mathfrak{nl}(T_{\mathfrak{nl},\ell}')=[\overline{\mathfrak{l}}]\psi(\overline{\mathfrak{l}}).$ The result now follows from \eqref{eq: 15a}.
\end{proof}
\end{lem}

\begin{thm}\label{thm 4.3.4}
  Let $\mathfrak{n},\mathfrak{nl}\in\mathcal{A}_c$ with $\ell=\mathfrak{l}\overline{\mathfrak{l}}\nmid N_\mathfrak{n} N_\Pi$ a split prime. Then, for any finite field extension $E/L$ and any $E$-valued character $\chi$ of $H_\mathfrak{n}$, we have 
    \begin{align}\label{eq: 11}
        \mathcal{N}^{\mathfrak{nl}}_\mathfrak{n}\left(\ _c\Xi_{\mathfrak{nl}}^{\Pi,\psi}\right)=P_\ell\left([\mathfrak{l}]\psi(\mathfrak{l})\ell^{-2-a-r}\right)\cdot\ _c\Xi_{\mathfrak{n}}^{\Pi,\psi}
    \end{align}
    as elements of $\mathbf{H}_\mathrm{mot}^5(\Pi,\psi,\mathfrak{n},L)_{E,\chi}.$
\end{thm}
\begin{proof}
    By definition, we have 
    \begin{align}
\scalemath{0.95}{\mathcal{N}^{\mathfrak{nl}}_\mathfrak{n}\left(_c\Xi_{\mathfrak{nl}}^{\Pi,\psi}\right)=[\overline{\mathfrak{l}}]^{-2}\psi(\overline{\mathfrak{l}})^{-2}\left\{ 1\otimes(\mathrm{pr}_1)_*\left(\ _c\Xi_\mathrm{mot}'([N_\Pi,N_\mathfrak{nl}])\right) \right\}-\left\{ [\mathfrak{l}]\psi(\mathfrak{l})\ell^{-1}\otimes (\mathrm{pr}_2)_*\left(\ _c\Xi_\mathrm{mot}'([N_\Pi,N_\mathfrak{nl}])\right)\right\} } 
        \end{align}
        where the class $_c\Xi_\mathrm{mot}'([N_\Pi,N_\mathfrak{nl}])$ is given by \Cref{def class Xi'}. Thus, in the quotient $\mathbf{H}_\mathrm{mot}^5(\Pi,\psi,\mathfrak{n},L)_{E,\chi}$, we have 
        \begin{align}\label{eq: 13}        \mathcal{N}^{\mathfrak{nl}}_\mathfrak{n}\left(_c\Xi_{\mathfrak{nl}}^{\Pi,\psi}\right)=1\otimes(\chi\psi)(\overline{\mathfrak{l}})^{-2}\left\{(\mathrm{pr}_1)_*\left(\ _c\Xi_\mathrm{mot}'([N_\Pi,N_\mathfrak{nl}])\right)-(\chi\psi)(\mathfrak{l})\ell^{-1}(\mathrm{pr}_2)_*\left(\ _c\Xi_\mathrm{mot}'([N_\Pi,N_\mathfrak{nl}])\right)\right\}.
        \end{align}As in the proof of \Cref{thm gl_2 norm-relations}, after fixing the local input data away from $\ell$, we have 
        \begin{align*}(\mathrm{pr}_1)_*\left(\ _c\Xi_\mathrm{mot}'([N_\Pi,N_\mathfrak{nl}])\right)&=\mathrm{Symbl}^{[a,b,d,r],\mathrm{sph}}_\ell\left(C_\ell\phi_{\ell,2}\otimes\left\{\xi_\ell^{\GSp_4}\otimes \mathrm{Tr}^{K_1(\ell)}_{\GL_2(\mathbf{Z}_\ell)}{\xi_{\ell,\mathbf{1}}^{\GL_2}}'\right\}\right)\\
        &=\mathrm{Symbl}^{[a,b,d,r],\mathrm{sph}}_\ell\left(C_\ell\phi_{\ell,2}\otimes\left\{\xi_\ell^{\GSp_4}\otimes \ell^{-1}(\ell-1)^{-1}\ch(\left[\begin{smallmatrix}
           & -1\\
           \ell^2 & 
       \end{smallmatrix}\right]\GL_2(\mathbf{Z}_\ell))\right\}\right).
        \end{align*}
         Recall the relation $(\mathrm{pr}_2)_*=S_\ell\cdot(\mathrm{pr}_2')_*$. The pushforward along the first map in the definition of $(\mathrm{pr}_2')_*$, corresponds to the action of $\left[\begin{smallmatrix}
             1 & \\
             & \ell^{-1}
         \end{smallmatrix}\right]$ on cohomology (see the proof of \cite[Proposition $6.1.5$]{grossi2020norm}). Thus, after once again fixing the data away from $\ell$, and using a similar commutative diagram as before, we have 
        $$(\mathrm{pr}_2)_*\left(\ _c\Xi_\mathrm{mot}'([N_\Pi,N_\mathfrak{nl}])\right)=S_\ell\cdot\mathrm{Symbl}^{[a,b,d,r],\mathrm{sph}}_\ell\left(C_\ell\phi_{\ell,2}\otimes\left\{\xi_\ell^{\GSp_4}\otimes \ell^{-1}(\ell-1)^{-1}\ch(\left[\begin{smallmatrix}
           & -\ell\\
           \ell^2 & 
\end{smallmatrix}\right]\GL_2(\mathbf{Z}_\ell))\right\}\right).$$ 
However in the quotient, $(\chi\psi)(\mathfrak{l})\ell^{-1}S_\ell$ is the same as $(\chi\psi)(\mathfrak{l})\ell^{-1}\omega(\ell)^{-1}\chi(\ell)^{-1}=\chi(\overline{\mathfrak{l}})^{-1}\psi(\overline{\mathfrak{l}})^{-1}=\ell^{-1/2}(\chi(\overline{\mathfrak{l}})\cdot \ell^{-1/2}\psi(\overline{\mathfrak{l}}))^{-1}.$
Combining these two facts with \Cref{def input at R}, we see that \eqref{eq: 13} is given in $\mathbf{H}_\mathrm{mot}^5(\Pi,\psi,\mathfrak{n},L)_{E,\chi}$, by
\begin{align}
    1\otimes(\mathrm{pr}_1)_*\left( _c\Xi_\mathrm{mot}(\mu_{{\ell}}[N_\Pi,N_\mathfrak{nl}] \right) 
\end{align}
where now $\ell$ acts as the distinguished prime $\ell_0$ in \Cref{def input at R}, and $\mu_\ell(\ell):=\ell^{-1/2}(\chi\psi)(\overline{\mathfrak{l}}).$ Naturally, we set $\nu_\ell(\ell):=\ell^{-1/2}(\chi\psi)(\mathfrak{l})$ and consider the unramified principal-series $I(\nu_\ell,\mu_\ell)$. Note that the spherical Hecke eigenvalues of $I(\nu_\ell,\mu_\ell)$ and $\mathcal{P}_\ell^\mathrm{spin}(\nu_\ell(\ell))$ are now both defined over $E$. Thus, by \Cref{thm gl_2 norm-relations}, we have that (notice the factor of $\ell^{-a-r}$ which pops up inside the brackets since we got rid of the twist $[-a-r]$ on the cohomology)
\begin{align}
    (\mathrm{pr}_1)_*\left( _c\Xi_\mathrm{mot}(\mu_{{\ell}}[N_\Pi,N_\mathfrak{nl}] \right)={\mathcal{P}_\ell^\mathrm{spin}}'(\ell^{-1/2}\ell^{-a-r}(\chi\psi)(\mathfrak{l}))\cdot\ _c\Xi_\mathrm{mot}([N_\Pi,N_\mathfrak{n}])
\end{align}
in $$\left[\mathbf{H}_\mathrm{mot}^5(N_\Pi,N_\mathfrak{n})_E\right]_{\substack{T_{\mathfrak{n},\ell}'=a_\ell\\
S_\ell'=b_\ell}},\ \ a_\ell:=(\chi\psi)(\mathfrak{l})+(\chi\psi)(\overline{\mathfrak{l}}),\ b_\ell:=(\chi\nu)(\ell).$$
Moreover, it follows by construction that $1\otimes (T_{\mathfrak{n},\ell}'-a_\ell)=1\otimes(S_\ell'-b_\ell)=0$ in $\mathbf{H}_\mathrm{mot}^5(\Pi,\psi,\mathfrak{n},L)_{E,\chi}$. Thus, also using \Cref{thm galois rep}, we see that \begin{align*}\mathcal{N}^{\mathfrak{nl}}_\mathfrak{n}\left(_c\Xi_{\mathfrak{nl}}^{\Pi,\psi}\right)&=1\otimes\left\{{\mathcal{P}_\ell^\mathrm{spin}}'(\ell^{-1/2}\ell^{-a-r}(\chi\psi)(\mathfrak{l}))\cdot \ _c\Xi_\mathrm{mot}'([N_\Pi,N_\mathfrak{n}])\right\}\\&=P_\ell([\mathfrak{l}]\psi(\mathfrak{l})\ell^{-2-a-r})\cdot (1\otimes _c\Xi_\mathrm{mot}'([N_\Pi,N_\mathfrak{n}]))\\
&=P_\ell([\mathfrak{l}]\psi(\mathfrak{l})\ell^{-2-a-r})\cdot\ _c\Xi_\mathfrak{n}^{\Pi,\psi}\end{align*}in $\mathbf{H}_\mathrm{mot}^5(\Pi,\psi,\mathfrak{n},L)_{E,\chi}$, concluding the proof.
\end{proof}
\begin{cor}\label{cor norm-relations over ray class groups}
    Let $\mathfrak{n},\mathfrak{nl}\in\mathcal{A}_c$ with $\ell=\mathfrak{l}\overline{\mathfrak{l}}\nmid N_\mathfrak{n} N_\Pi$ a split prime. Then, 
    \begin{align}\label{eq: 16}\mathcal{N}^{\mathfrak{nl}}_\mathfrak{n}\left(\ _c\Xi_{\mathfrak{nl}}^{\Pi,\psi}\right)=P_\ell\left([\mathfrak{l}]\psi(\mathfrak{l})\ell^{-2-a-r}\right)\cdot\ _c\Xi_{\mathfrak{n}}^{\Pi,\psi}\end{align}
    as elements of $\mathbf{H}_\mathrm{mot}^5(\Pi,\psi,\mathfrak{n},L)$.
\end{cor}
\begin{proof}
    By \Cref{thm 4.3.4}, we can enlarge $E$ if necessary to contain all roots of unity of order equal to the size of $H_\mathfrak{n}.$ Then, \eqref{eq: 16} holds in $\mathbf{H}_\mathrm{mot}^5(\Pi,\psi,\mathfrak{n},L)_{E,\chi}$ for all characters of $H_\mathfrak{n}$. Also note that $\mathbf{H}_\mathrm{mot}^5(\Pi,\psi,\mathfrak{n},L)_{E}$ splits as a sum of eigenspaces over all characters of $H_\mathfrak{n}$. Combining these two facts, we see that \eqref{eq: 16} holds in $\mathbf{H}_\mathrm{mot}^5(\Pi,\psi,\mathfrak{n},L)_{E}$. Finally, since both sides of \eqref{eq: 16} are defined over $L$ and since $$\mathbf{H}_\mathrm{mot}^5(\Pi,\psi,\mathfrak{n},L)_{E}=\mathbf{H}_\mathrm{mot}^5(\Pi,\psi,\mathfrak{n},L)\otimes_L E$$
    the result follows by flat base change.
\end{proof}
\subsection{Mapping to Galois cohomology}\label{sec: mapping to Galois}

Recall that $\Pi$ denotes a general-type cuspidal automorphic representation of $\GSp_4$. We once again enlarge $L$, if necessary, to a finite Galois extension of $\mathbf{Q}$ over which $\Pi_\mathrm{f}$ is definable, as in \cite[\S 2]{loeffler2024universaleulergsp4}. We write $\Pi'$ for the (non-unitary) twist $\Pi\otimes|\mu(-)|_\mathbf{A}^{(3-w)/2}$. From now on, 
\begin{itemize}
    \item Let $p$ be a prime such that $p\nmid N_\Pi N_\mathfrak{f}$ and let $\mathfrak{P}|\mathfrak{p}$ be primes of $L|K$ above $p$.
    \item Let $d=0$, and hence $r=0$, $a=k_2-3$ and $b=k_1-k_2$. 
\end{itemize}
  Consider the natural map
\begin{align}\label{eq: basechange}
    H^5_\mathrm{\acute{e}t}(Y_G, \mathscr{W}_{\mathbf{Q}_p}^{a,b,0,*}(3-a))_{L_\mathfrak{P}}\rightarrow H^0\left(\mathbf{Q}, H^5_\mathrm{\acute{e}t}(Y_{G,\overline{\mathbf{Q}}},\mathscr{W}_{\mathbf{Q}_p}^{a,b,0,*}(3-a))_{L_\mathfrak{P}}\right)
\end{align}
where on the left we use continuous \'etale cohomology in the sense of \cite{jannsen1988continuous} and \cite{loeffler2021euler}. For a $\mathcal{H}_{L,N_\Pi}^{\GSp_4}$-module $M$, we write $M(\Pi_\mathrm{f}'^\vee)$ for the generalized eigenspace with respect to the spherical Hecke eigensystem of $\Pi_\mathrm{f}'^\vee$. We now claim that $H^5_\mathrm{\acute{e}t}(Y_G, \mathscr{W}_{\mathbf{Q}_p}^{a,b,0,*}(3-a))_{L_\mathfrak{P}}(\Pi_\mathrm{f}'^\vee)$ is contained in the homologically trivial classes (the kernel of \eqref{eq: basechange}): For ease of notation, we temporarily drop the subscript $L_\mathfrak{P}$. Under \eqref{eq: basechange}, an element in $H^5_\mathrm{\acute{e}t}(Y_G, \mathscr{W}_{\mathbf{Q}_p}^{a,b,0,*}(3-a))(\Pi_\mathrm{f}'^\vee)$ is mapped to $H^0(\mathbf{Q}, H^5_\mathrm{\acute{e}t}(Y_{G,\overline{\mathbf{Q}}},\mathscr{W}_{\mathbf{Q}_p}^{a,b,0,*}(3-a))(\Pi_\mathrm{f}'^\vee))$. By Kunneth formula for \'etale cohomology and the fact that direct limits commute with direct sums and tensor products, we have
\begin{align}\label{eq: 23}
    H^5_\mathrm{\acute{e}t}(Y_{G,\overline{\mathbf{Q}}},\mathscr{W}_{\mathbf{Q}_p}^{a,b,0,*}(3-a))(\Pi_\mathrm{f}'^\vee)=\bigoplus_{i+j=5} H^i_{\mathrm{\acute{e}t}}(Y_{\mathrm{GSp}_4,\overline{\mathbf{Q}}},\mathscr{V}_{\mathbf{Q}_p}^{a,b,*}(3-a))(\Pi_\mathrm{f}'^\vee)\otimes H^j_{\mathrm{\acute{e}t}}(Y_{\GL_2,\overline{\mathbf{Q}}},\mathbf{Q}_p)
\end{align}
where $\mathscr{V}_{\mathbf{Q}}^{a,b}$ is the $\GSp_4(\mathbf{A}_\mathbf{f})$-equivariant relative Chow motive over $Y_{\GSp_4}$ associated to the highest weight algebraic $\GSp_4$-representation $V^{a,b}$ under Ancona's functor \cite{ancona2015decomposition}. For $j\geq 2$, $H^j_{\mathrm{\acute{e}t}}(Y_{\GL_2,\overline{\mathbf{Q}}},\mathbf{Q}_p)=0$ by the Artin vanishing theorem. Moreover, by \cite[\S 1]{weissauer2005four}, the terms $H^i_{\mathrm{\acute{e}t}}(Y_{\mathrm{GSp}_4,\overline{\mathbf{Q}}},\dots)(\Pi_\mathrm{f}'^\vee)$ are zero for $i=4,5$. Thus, we conclude \eqref{eq: 23} is identically zero which proves the claim. 

Combining this with the edge map from the Hochschild-Serre spectral sequence, the projection map onto the $\Pi_\mathrm{f}'^\vee$-eigenspace
$$H^1\left(\mathbf{Q},H^4_\mathrm{\acute{e}t}(Y_{G,\overline{\mathbf{Q}}},\mathscr{W}_{\mathrm{Q}_p}^{a,b,0,*}(3-a))_{L_\mathfrak{P}}\right)\rightarrow H^1\left(\mathbf{Q},H^4_\mathrm{\acute{e}t}(Y_{G,\overline{\mathbf{Q}}},\mathscr{W}_{\mathbf{Q}_p}^{a,b,0,*}(3-a))_{L_\mathfrak{P}}[\Pi_\mathrm{f}'^\vee]\right)$$
lifts to a map 
\begin{align}\label{eq: 28}
\mathrm{AJ}^{\Pi,a,b,r}_{L_\mathfrak{P}}:H^5_\mathrm{\acute{e}t}(Y_G,\mathscr{W}_{\mathbf{Q}_p}^{a,b,0,*}(3-a))_{L_\mathfrak{P}}\rightarrow H^1(\mathbf{Q},H^4_\mathrm{\acute{e}t}(Y_{G,\overline{\mathbf{Q}}},\mathscr{W}_{\mathbf{Q}_p}^{a,b,0,*}(3-a))_{L_\mathfrak{P}}[\Pi_\mathrm{f}'^\vee])
\end{align}
characterized as the unique Hecke-equivariant map agreeing with the $\Pi_\mathrm{f}'^\vee$-projection of the \'etale Abel-Jacobi map on homologically trivial classes. We note that this is the exact same lifting strategy used in \cite[\S $3.4.2$]{loeffler2024universaleulergsp4}. The map \eqref{eq: 28} is by construction compatible with taking finite level on both sides. Using Kunneth formula one more time for the $H^4_\mathrm{\acute{e}t}$ term, we obtain a Hecke equivariant map which we again denote by 
\begin{align}
\scalemath{0.95}{\mathrm{AJ}^{\Pi,a,b,r}_{L_\mathfrak{P}}:H^5_\mathrm{\acute{e}t}(Y_G,\mathscr{W}_{\mathbf{Q}_p}^{a,b,0,*}(3-a))_{L_\mathfrak{P}}\rightarrow H^1\left(\mathbf{Q},H^3_\mathrm{\acute{e}t}(Y_{\mathrm{GSp_4},\overline{\mathbf{Q}}},\mathscr{V}_{\mathrm{Q}_p}^{a,b,*}(3-a))_{L_\mathfrak{P}}[\Pi_\mathrm{f}'^\vee]\otimes_{L_\mathfrak{P}} H^1_\mathrm{\acute{e}t}(Y_{\GL_2,\overline{\mathbf{Q}}},\mathbf{Q}_p)_{L_\mathfrak{P}}\right)}.
\end{align}
Let $\mathcal{M}(\Pi_\mathrm{f}')$ be an arbitrary but fixed $L$-model of $\Pi_\mathrm{f}'$, i.e. $\mathcal{M}(\Pi_\mathrm{f}')\otimes_L\mathbf{C}\simeq \Pi_\mathrm{f}'$ as $\GSp_4(\mathbf{A}_\mathrm{f})$-representations. As in \cite[Definition $3.3.2$]{loeffler2024universaleulergsp4}, we set 
$$V_\Pi:=\mathrm{Hom}_{L_\mathfrak{P}[\GSp_4(\mathbf{A}_\mathrm{f})]}\left(\mathcal{M}(\Pi_\mathrm{f}')_{L_\mathfrak{P}},H^3_{\mathrm{\acute{e}t},c}(Y_{\GSp_4,\overline{\mathbf{Q}}},\mathscr{V}_{\mathbf{Q}_p}^{a,b})_{L_\mathfrak{P}}\right).$$
Since by assumption, $\Pi$ is of general-type, $V_\Pi$ is a four dimensional $L_\mathfrak{P}$-linear $\mathrm{Gal}(\overline{\mathbf{Q}}/\mathbf{Q})$-representation, whose semi-simplification is isomorphic to $\rho_{\Pi,p}$ of \Cref{thm galois rep}. By \cite[Proposition $3.3.4$]{loeffler2024universaleulergsp4} there is a canonical isomorphism of $\GSp_4(\mathbf{A}_\mathrm{f})$-representations
\begin{align}\label{eq: 26}
    \mathcal{M}(\Pi_\mathrm{f}')_{L_\mathfrak{P}}\simeq \mathrm{Hom}_{\mathrm{Gal}(\overline{\mathbf{Q}}/\mathbf{Q})}\left(H^3_\mathrm{\acute{e}t}(Y_{\GSp_4,\overline{\mathbf{Q}}},\mathscr{V}_{\mathbf{Q}_p}^{a,b,*}(3))_{L_\mathfrak{P}}[\Pi_\mathrm{f}'^\vee],V_\Pi^* \right),\ \ v\mapsto \Phi_v
\end{align}
induced by the infinite level $\GSp_4(\mathbf{A}_\mathrm{f})\times\mathrm{Gal}(\overline{\mathbf{Q}}/\mathbf{Q})$ Poincar\'e duality pairing
$$\langle\cdot,\cdot\rangle:H^3_{\mathrm{\acute{e}t},c}(Y_{\GSp_4,\overline{\mathbf{Q}}},\mathscr{V}_{\mathbf{Q}_p}^{a,b})_{L_\mathfrak{P}}\times H^3_\mathrm{\acute{e}t}(Y_{\GSp_4,\overline{\mathbf{Q}}},\mathscr{V}_{\mathbf{Q}_p}^{a,b,*}(3))_{L_\mathfrak{P}}\rightarrow L_\mathfrak{p.}$$
In other words, if $v\in\mathcal{M}(\Pi_\mathrm{f}'), x\in H^3_\mathrm{\acute{e}t}(Y_{\GSp_4,\overline{\mathbf{Q}}},\mathscr{V}_{\mathbf{Q}_p}^{a,b,*}(3))_{L_\mathfrak{P}}[\Pi_\mathrm{f}'^\vee]$ and $f\in V_\Pi$, then the isomorphism \eqref{eq: 26} is characterized by $\Phi_v(x)(f)=\langle f(v),x\rangle$.

Recall that the subgroup $U^{\GSp_4}(N_\Pi)\subseteq\GSp_4(\hat{\mathbf{Z}})$ is an open compact level subgroup, unramified away from $N_\Pi:=\prod_{\ell\in S}\ell$, where $S$ is a finite set of primes containing all primes at which $\Pi$ ramifies. By enlarging $S$ and shrinking $U^{\GSp_4}(N_\Pi)$ is necessary, we may assume that $\Pi_\mathrm{f}$ has non-zero $U^{\GSp_4}(N_\Pi)$-invariants. We fix a choice of non-zero $U^{\GSp_4}(N_\Pi)$-invariant vector $v_0\in \mathcal{M}(\Pi_\mathrm{f}')^{U^{\GSp_4}(N_\Pi)}$, which corresponds to a modular parametrisation $\Phi_{v_0}$ under \eqref{eq: 26}. We denote by $T_\Pi^*$ the image of $H^3_\mathrm{\acute{e}t}(Y_{\GSp_4,\overline{\mathbf{Q}}},\mathscr{V}_{\mathbf{Z}_p}^{a,b,*}(3))_{\mathcal{O}_{L_\mathfrak{P}}}$ under $\Phi_{v_0}$. Then $T_\Pi^*$ is a Galois-stable $\mathcal{O}_{L_\mathfrak{P}}$-lattice in $V_\Pi^*.$

\begin{rem} If $\Pi_\mathrm{f}$ is assumed to be everywhere locally generic, we can take $\mathcal{M}(\Pi_\mathrm{f}')$   to be the $L$-linear Whittaker model of $\Pi_\mathrm{f}'$ in the sense of  \cite[Definition $10.2$]{loeffler2021higher} and $U^{\GSp_4}(N_\Pi)$ can be taken to be the paramodular subgroup of level $N_\Pi$. Then, by the new-vector theory of \cite{roberts2007local} and \cite{okazaki2019localwhittakernewformsgsp4matching}, there is a unique choice of non-zero new-vector $v_0$.
\end{rem}
The composition of $\mathrm{AJ}_{L_\mathfrak{P}}^{\Pi,a,b,r}$ with $\Phi_{v_0,*}$ (the induced map on $H^1(\mathbf{Q},\dots)$) gives a linear map 
\begin{align}
\Phi_{v_0,*}\circ\mathrm{AJ}_{L_\mathfrak{P}}^{\Pi,a,b,r} : H^5_\mathrm{\acute{e}t}(Y_G,\mathscr{W}_{\mathbf{Q}_p}^{a,b,0,*}(3-a))_{L_\mathfrak{P}}&\rightarrow H^1\left(\mathbf{Q},V_\Pi^*(-a)\otimes_{L_\mathfrak{P}} H^1_\mathrm{\acute{e}t}(Y_{\GL_2,\overline{\mathbf{Q}}},\mathbf{Q}_p)_{L_\mathfrak{P}}\right)\\
\nonumber &\overset{\simeq}{\rightarrow}  H^1\left(\mathbf{Q},V_\Pi^*(-a-1)\otimes_{L_\mathfrak{P}} H^1_\mathrm{\acute{e}t}(Y_{\GL_2,\overline{\mathbf{Q}}},\mathbf{Q}_p)(1)_{L_\mathfrak{P}}\right)
\end{align}
\begin{defn}
    We write $\mathcal{A}_{cp}$ for the set of integral ideals in $\mathcal{A}_c$ \emph{(\emph{\Cref{def set A and norm map}})} that satisfy the extra condition of having norm coprime to $p$. 
\end{defn}
\begin{defn}\label{def H^1_et}
   Let $\mathfrak{n}$ be an ideal in $\mathcal{A}_{cp}$. Similarly to \emph{\Cref{def quot of H5}}, we set 
   $$H^1_\mathrm{\acute{e}t}(\psi,\mathfrak{n},L_\mathfrak{P}):= L_\mathfrak{P}[H_\mathfrak{n}]\bigotimes_{\left(\phi_\mathfrak{n},\mathbf{T}_{L_\mathfrak{P}, N_\mathfrak{n}}^{\GL_2}\right)} H^1_\mathrm{\acute{e}t}(Y_{\GL_2}(N_\mathfrak{n})_{\overline{\mathbf{Q}}},\mathbf{Q}_p(1))_{L_\mathfrak{P}}.$$
\end{defn}

\begin{prop}\label{prop descent of AJ}Let $\mathfrak{n}$ be an ideal in $\mathcal{A}_{cp}$ and write $r_\mathrm{\acute{e}t}$ for the $p$-adic \'etale regulator map as in \cite{jannsen1988continuous}. The composition $\Phi_{v_0,*}\circ \mathrm{AJ}_{L_\mathfrak{P}}^{\Pi,a,b,r}\circ r_{\mathrm{\acute{e}t},L_\mathfrak{P}}$ descents to a well-defined map
    $$\Phi_{v_0,*}\circ \mathrm{AJ}_{L_\mathfrak{P}}^{\Pi,a,b,r}\circ r_{\mathrm{\acute{e}t},L_\mathfrak{P}}: H^5_\mathrm{mot}(\Pi,\psi,\mathfrak{n},L)\rightarrow H^1\left(\mathbf{Q},V_\Pi^*(-a-1)\otimes_{L_\mathfrak{P}} H^1_\mathrm{\acute{e}t}(\psi,\mathfrak{n},L_\mathfrak{P})\right)$$
    \begin{proof}
        Recall from \Cref{def quot of H5} that $$H^5_\mathrm{mot}(\Pi,\psi,\mathfrak{n},L)=L[H_\mathfrak{n}]\underset{\left(\Theta_{\Pi^\vee}[\tfrac{w-3}{2}]\otimes\phi_\mathfrak{n},\mathcal{H}_{L,N_\Pi}^{\GSp_4}\otimes_L\mathbf{T}_{L,N_\mathfrak{n}}^{\GL_2}\right)}{\bigotimes}\mathbf{H}_\mathrm{mot}^5(N_\Pi,N_\mathfrak{n})_{L}.$$
         The descent over the $\GL_2$ Hecke algebra with respect to $\phi_\mathfrak{n}$, follows by construction, the compatibility of $\mathrm{AJ}_{L_\mathfrak{P}}^{\Pi,a,b,r}$ and $r_{\mathrm{\acute{e}t},L_\mathfrak{P}}$ with finite levels, and the Hecke equivariance of these two maps. It remains to verify the descent over the unramified $\GSp_4$ Hecke algebra. Let $\theta$ be a Hecke operator in the unramified Hecke algebra $\mathcal{H}_{L,N_\Pi}^{\GSp_4}$ and let $x\in H^5_\mathrm{mot}(N_\Pi,N_\mathfrak{n})_{L}$. Again, by the Hecke equivariance of $\mathrm{AJ}_{L_\mathfrak{P}}^{\Pi,a,b,r}$ and $r_{\mathrm{\acute{e}t},L_\mathfrak{P}}$, the image of $\theta\cdot x$ under $\Phi_{v_0,*}\circ \mathrm{AJ}_{L_\mathfrak{P}}^{\Pi,a,b,r}\circ r_{\mathrm{\acute{e}t},L_\mathfrak{P}}$ is given by $\Phi_{v_0,*}(\theta\cdot (\mathrm{AJ}_{L_\mathfrak{P}}^{\Pi,a,b,r}(r_{\mathrm{\acute{e}t},L_\mathfrak{P}}(x)))=:C.$ Regarding the latter as a cocycle and evaluating at an element $g\in\mathrm{Gal}(\overline{\mathbf{Q}}/\mathbf{Q})$, we have $C(g)=\Phi_{v_0}(\theta\cdot\mathrm{AJ}_{L_\mathfrak{P}}^{\Pi,a,b,r}(r_{\mathrm{\acute{e}t},L_\mathfrak{P}}(x))(g))$. By linearity, we may assume that $\mathrm{AJ}_{L_\mathfrak{P}}^{\Pi,a,b,r}(r_{\mathrm{\acute{e}t},L_\mathfrak{P}}(x))(g)$ is a pure tensor $$y_1\otimes y_2\in H^3_\mathrm{\acute{e}t}(Y_{\mathrm{GSp_4},\overline{\mathbf{Q}}},\mathscr{V}_{\mathrm{Q}_p}^{a,b,*}(3-a))_{L_\mathfrak{P}}[\Pi_\mathrm{f}'^\vee]\otimes_{L_\mathfrak{P}} H^1_\mathrm{\acute{e}t}(Y_{\GL_2,\overline{\mathbf{Q}}},\mathbf{Q}_p)_{L_\mathfrak{P}}.$$ 
         Hence $C(g)=\Phi_{v_0}(\theta\cdot(y_1\otimes y_2))=\Phi_{v_0}(\theta\cdot y_1)\otimes y_2$ since the $\GSp_4$ Hecke action only takes place on the $H^3_\mathrm{\acute{e}t}$ component.
         By the characterizing property of \eqref{eq: 26}, upon evaluating on any $f\in V_\Pi$ we see that $\Phi_{v_0}(\theta\cdot y_1)(f)=\left\langle f(v_0),\theta\cdot y_1\right\rangle.$
         By construction, all three vectors $v_0,f(v_0)$ and $y_1$ are invariant under the action of the subgroup $U^{\GSp_4}(N_\Pi).$ In particular, they are invariant under $\GSp_4(\hat{\mathbf{Z}}^S)$, and $\theta$ is by definition an element of $C_c^\infty(\GSp_4(\hat{\mathbf{Z}}^S)\backslash \GSp_4(\mathbf{A}_\mathrm{f}^S)/\GSp_4(\hat{\mathbf{Z}}^S))$. Thus, by the $\GSp_4(\mathbf{A}_\mathrm{f})$-equivariance of the Poincar\'e pairing, and the definition of $V_\Pi$, we have 
         \begin{align*}
             \Phi_{v_0}(\theta\cdot y_1)(f)&=\left\langle f(v_0),\theta\cdot y_1\right\rangle\\
             &=\left\langle f(\theta'\cdot v_0),y_1\right\rangle\\
             &=\Theta_{\Pi'}(\theta')\cdot \left\langle f(v_0),y_1\right\rangle\\
             &=\Theta_{\Pi'}(\theta')\Phi_{v_0}( y_1)(f)
        \end{align*}
         where the last equality follows from the fact $\Pi$ is spherical away from $N_\Pi$, and $\Theta_{\Pi'}$ denotes the spherical Hecke eigensystem of $\Pi'=\Pi\otimes |\mu(-)|_\mathbf{A}^{(3-w)/2}$. But then, $\Theta_{\Pi'}(\theta')$ is nothing more than $\Theta_{\Pi^\vee}[\tfrac{w-3}{2}](\theta)$, which is the map appearing in the definition of $H^5_\mathrm{mot}(\Pi,\psi,\mathfrak{n},L)$. 
    \end{proof}
\end{prop}

\begin{defn}
    For $\mathfrak{n}\in\mathcal{A}_{cp}$, we define the class 
    $$_c \tilde{z}_\mathfrak{n}^{\Pi,\psi}\in H^1\left(\mathbf{Q},V_\Pi^*(-a-1)\otimes_{L_\mathfrak{P}} H^1_\mathrm{\acute{e}t}(\psi,\mathfrak{n},L_\mathfrak{P})\right)$$ to be the image of the class $_c \Xi_\mathfrak{n}^{\Pi,\psi}$ under the map of \emph{\Cref{prop descent of AJ}}.
\end{defn}
\begin{rem}\label{rem integrality}By \Cref{prop: integrality} and the fact that everything now is defined over $L_\mathfrak{P}$, we see that the classes $_c \tilde{z}_\mathfrak{n}^{\Pi,\psi}$ lie in the image of $H^1(\mathbf{Q}, T_\Pi^*(-a-1)\otimes_{\mathcal{O}_{L_\mathfrak{P}}} H^1_\mathrm{\acute{e}t}(\psi,\mathfrak{n},\mathcal{O}_{L_\mathfrak{P}}))$ where $H^1_\mathrm{\acute{e}t}(\psi,\mathfrak{n},\mathcal{O}_{L_\mathfrak{P}})$ is defined as in \Cref{def H^1_et} but replacing $L_\mathfrak{P}$ with $\mathcal{O}_{L_\mathfrak{P}}$ and $\mathbf{Q}_p$ with $\mathbf{Z}_p.$
\end{rem}
\subsection{CM-patching ala Lei-Loeffler-Zerbes}\label{sec: CM patching} Let $\mathfrak{n}$ be an integral ideal in $\mathcal{A}_{cp}.$
\begin{defn}[{\cite[Definition $5.2.1$ \& $5.2.2$]{Lei_Loeffler_Zerbes_2015}}]\label{def H^1p}Let $H_\mathfrak{n}^{(p)}$ be the largest quotient of $H_\mathfrak{n}$ whose order is a power of $p$, and let 
$$\phi_\mathfrak{n}^{(p)}:\mathbf{T}_{L_\mathfrak{P},N_\mathfrak{n}}^{\GL_2}\longrightarrow L_\mathfrak{P}[H_\mathfrak{n}^{(p)}]$$
be the composition of the map $\phi_\mathfrak{n}$ and the natural projection $H_\mathfrak{n}\rightarrow H_\mathfrak{n}^{(p)}.$ Moreover, for $R\in\{L_\mathfrak{P},\mathcal{O}_{L_\mathfrak{P}}\}$, let 
$$H^{1}(\psi,\mathfrak{n},R):=R[H_\mathfrak{n}^{(p)}]\bigotimes_{\left(\phi_\mathfrak{n}^{(p)},\mathbf{T}_{R,N_\mathfrak{n}}^{\GL_2}\right)} H^1_\mathrm{\acute{e}t}(Y_{\GL_2}(N_\mathfrak{n})_{\overline{\mathbf{Q}}},A_R(1))_{R}$$
where $A_R:=\mathbf{Q}_p$ if $R=L_\mathfrak{P}$, and $\mathbf{Z}_p$ if $R=\mathcal{O}_{L_\mathfrak{P}}.$
\end{defn}

\begin{defn}\label{def: classe over H_n^(p)}
    For $\mathfrak{n}\in\mathcal{A}_{cp}$, we define the class 
    $$_c z_\mathfrak{n}^{\Pi,\psi}\in H^1\left(\mathbf{Q},V_\Pi^*(-a-1)\otimes_{L_\mathfrak{P}} H^{1}(\psi,\mathfrak{n},L_\mathfrak{P})\right)$$
    to be the image of the class $_c \tilde{z}_\mathfrak{n}^{\Pi,\psi}$ under the map induced from the natural projection $H_\mathfrak{n}\rightarrow H_\mathfrak{n}^{(p)}$.
\end{defn}
 Let $\mathfrak{n},\mathfrak{nl}\in\mathcal{A}_{cp}$ with $\ell=\mathfrak{l}\overline{\mathfrak{l}}\nmid \mathfrak{n}$ a split prime. In the exact same way as \Cref{lem intertwining action}, the norm map of \Cref{def set A and norm map}, descends to 
$$\mathcal{N}^{\mathfrak{nl}}_\mathfrak{n}:H^{1}(\psi,\mathfrak{nl},R)\longrightarrow H^{1}(\psi,\mathfrak{n},R)$$
for $R\in\{L_\mathfrak{P},\mathcal{O}_{L_\mathfrak{P}}\}$.
\begin{cor}\label{cor norm-relations one but final}
 Let $\mathfrak{n},\mathfrak{nl}\in\mathcal{A}_{cp}$ with $\ell=\mathfrak{l}\overline{\mathfrak{l}}\nmid N_\mathfrak{n} N_\Pi$ a split prime. Then,
$$\mathcal{N}^{\mathfrak{nl}}_\mathfrak{n}\left(\ _cz_{\mathfrak{nl}}^{\Pi,\psi}\right)=P_\ell\left([\mathfrak{l}]\psi(\mathfrak{l})\ell^{-2-a}\right)\cdot\ _c z_{\mathfrak{n}}^{\Pi,\psi}$$
as an equality of elements in $H^1(\mathbf{Q},V_\Pi^*(-a-1)\otimes_{L_\mathfrak{P}} H^{1}(\psi,\mathfrak{n},L_\mathfrak{P}))$. Moreover, the classes $_cz_{\mathfrak{n}}^{\Pi,\psi}$ for any $\mathfrak{n}\in\mathcal{A}_{cp}$ lie in the image of $H^1(\mathbf{Q}, T_\Pi^*(-a-1)\otimes_{\mathcal{O}_{L_\mathfrak{P}}} H^{1}(\psi,\mathfrak{n},\mathcal{O}_{L_\mathfrak{P}}))$.
\begin{proof}
    The norm relation follows from \Cref{cor norm-relations over ray class groups}, whereass the integrality statement follows from \Cref{rem integrality}.
\end{proof}
\end{cor} However, the classes $_cz_\mathfrak{n}^{\Pi,\psi}$ still live on different modular curves as we vary $\mathfrak{n}$. We will now apply the CM-patching method of \cite[\S $5.2$]{Lei_Loeffler_Zerbes_2015} to merge them together and obtain Euler system norm-relations over an infinite tower of ray class fields of our imaginary quadratic field $K$. We now summarize the results of \cite[\S$5.2$]{Lei_Loeffler_Zerbes_2015}, translated to our setting. Firstly, the choice of the primes $\mathfrak{P}|\mathfrak{p}|p$ of $L/K/\mathbf{Q}$, allows us to fix once and for all the following commutative diagram of embeddings of fields
\begin{align}\label{eq: diagram of fields}\begin{tikzcd}[ampersand replacement=\&,cramped,column sep=scriptsize,row sep=small]
	L \& {L_\mathfrak{P}} \& {\overline{\mathbf{Q}}_p} \\
	K \& {K_\mathfrak{p}}
	\arrow[hook, from=1-1, to=1-2]
	\arrow[hook, from=1-2, to=1-3]
	\arrow[hook, from=2-1, to=1-1]
	\arrow[hook, from=2-1, to=2-2]
	\arrow[hook, from=2-2, to=1-2]
	\arrow[hook, from=2-2, to=1-3]
\end{tikzcd}\end{align}
which we will always implicitly use for the rest of the paper when working inside $\overline{\mathbf{Q}}_p.$ We define the continuous character $$\psi_\mathfrak{P}:K^\times\backslash\mathbf{A}_{K,\mathrm{f}}^\times \longrightarrow L_\mathfrak{P}^\times,\ \ \ x\mapsto x_\mathfrak{p}^{-1}\psi_\mathrm{f}(x).$$
This is well-defined since by assumption $\psi$ has infinity type $(-1,0)$. We regard $\psi_\mathfrak{P}$ as a character of $\mathrm{Gal}(\overline{K}/K)$ via class field theory, where we normalize the global Artin map $K^\times\backslash\mathbf{A}_K^\times \rightarrow \mathrm{Gal}(\overline{K }/K)^\mathrm{ab}$ to map uniformizers to geometric Frobenius elements. Let $R\in\{L_\mathfrak{P},\mathcal{O}_{L_\mathfrak{P}}\}$. 

\begin{enumerate}
    \item Let $\mathfrak{n}\in\mathcal{A}_{cp}$. Then the $R[H_\mathfrak{n}^{(p)}]$-module $H^{1}(\psi,\mathfrak{n},R)$ is free of rank $2$: This is \cite[Proposition $5.2.3$]{Lei_Loeffler_Zerbes_2015}.
    \item Let $\mathfrak{n}\in\mathcal{A}_{cp}$. Then the module $H^{1}(\psi,\mathfrak{n},L_\mathfrak{P})$ is isomorphic as a $L_\mathfrak{P}[H_\mathfrak{n}^{(p)}]$-linear $\mathrm{Gal}(\overline{\mathbf{Q}}/\mathbf{Q})$-module to the induced representation $\mathrm{Ind}_{K(\mathfrak{n})}^\mathbf{Q}(L_\mathfrak{P}(\psi_\mathfrak{P}))^*$, where $K(\mathfrak{n})$ denotes the ray class field corresponding to $H_\mathfrak{n}^{(p)}.$ This is \cite[Theorem $5.2.4$]{Lei_Loeffler_Zerbes_2015}. The proof is the same after verifying that in our adelic setup, and for any character $\eta$ of $H_\mathfrak{n}^{(p)}$, the $L_\mathfrak{P}$-vector space $$L_\mathfrak{P}\bigotimes_{\left(\eta,L_\mathfrak{P}[H_\mathfrak{n}^{(p)}]\right)}H^{1}(\psi,\mathfrak{n},L_\mathfrak{P})$$
    is identified with the largest quotient of $H^1_\mathrm{\acute{e}t}(Y_{\GL_2}(N_\mathfrak{n})_{\overline{\mathbf{Q}}},\mathbf{Q}_p(1))_{L_\mathfrak{P}}$ on which the Hecke operators $T_{\mathfrak{n},\ell}'$ for all primes $\ell$, act as multiplication by the $\ell$-th Fourier coefficient of the eigenform $g_{\psi \eta}$. This follows by construction.
    \item  Let $\mathfrak{n},\mathfrak{nl}\in\mathcal{A}_{cp}$ with $\ell=\mathfrak{l}\overline{\mathfrak{l}}\nmid \mathfrak{n}$ a split prime and let $R\in\{L_\mathfrak{p},\mathcal{O}_{L_\mathfrak{P}}\}$. The map  
$$R[H_\mathfrak{n}^{(p)}]\otimes_{R[H_\mathfrak{ln}^{(p)}]} H^{1}(\psi,\mathfrak{ln},R)\longrightarrow H^{1}(\psi,\mathfrak{n},R),\ \ f(h)\otimes x\mapsto f(h)\cdot\mathcal{N}^\mathfrak{ln}_\mathfrak{n}(x)$$
is an isomorphism. This is part of \cite[Proposition $5.2.5$]{Lei_Loeffler_Zerbes_2015}, together with the fact that our extra normalization factor $[\overline{\mathfrak{l}}]^{-2}\psi(\overline{\mathfrak{l}})^{-2}$ in the definition of $\mathcal{N}^{\mathfrak{ln}}_\mathfrak{n}$, is a unit in $\mathcal{O}_{L_\mathfrak{P}}[H_\mathfrak{n}^{(p)}].$
\item\label{eq: item 4} One can find a family of (non-canonical) isomorphisms 
\begin{align*}\label{eq: fam of isos}\mathfrak{\nu_\mathfrak{n}}:H^{1}(\psi,\mathfrak{n},R)\simeq\mathrm{Ind}_{K(\mathfrak{n})}^\mathbf{Q} R(\psi_\mathfrak{P}^{-1}),\ \ \forall\  \mathfrak{n}\in\mathcal{A}_{cp}\end{align*}
of $R[H_\mathfrak{n}^{(p)}][\mathrm{Gal}(\overline{\mathbf{Q}}/\mathbf{Q})]$-modules, such that for any $\mathfrak{n},\mathfrak{nl}\in\mathcal{A}_{cp}$ with $\ell=\mathfrak{l}\overline{\mathfrak{l}}\nmid \mathfrak{n} $ a split prime, the diagram 
\[\begin{tikzcd}[ampersand replacement=\&]
	{H^{1}(\psi,\mathfrak{nl},R)} \& {\mathrm{Ind}_{K(\mathfrak{nl})}^\mathbf{Q} R(\psi_\mathfrak{P}^{-1})} \\
	{H^{1}(\psi,\mathfrak{n},R)} \& {\mathrm{Ind}_{K(\mathfrak{n})}^\mathbf{Q} R(\psi_\mathfrak{P}^{-1})}
	\arrow["{\simeq_{\nu_\mathfrak{nl}}}", from=1-1, to=1-2]
	\arrow["{\mathcal{N}^\mathfrak{nl}_\mathfrak{n}}"', from=1-1, to=2-1]
	\arrow[from=1-2, to=2-2]
	\arrow["{\simeq_{\nu_\mathfrak{n}}}", from=2-1, to=2-2]
\end{tikzcd}\]
commutes. This is part of \cite[Corollary $5.2.6$]{Lei_Loeffler_Zerbes_2015}, and uses the three prior results mentioned above. We use the same notation $\nu_\mathfrak{n}$ for both $R=\mathcal{O}_{L_\mathfrak{P}}$ and $R=L_\mathfrak{P}$ to signify that the latter is taken to be the base change of the former. We also note that our $\nu_\mathfrak{n}$ are not exactly the same as the ones in \textit{loc.cit.}, since as mentioned above, the maps $\mathcal{N}_\mathfrak{n}^{\mathfrak{nl}}$ differ by a normalization factor from the ones in \textit{loc.cit.}. The way the $\nu_\mathfrak{n}$ are obtained however, is exactly the same as in \textit{loc.cit.}.
\end{enumerate}
\subsubsection{Galois cohomology classes over ray class fields} We now have all the ingredients needed to define classes in the Galois cohomology over ray class fields and show that they are norm compatible. We fix a family of isomorphisms $\{\nu_\mathfrak{n}\ |\ \mathfrak{n}\in\mathcal{A}_{cp}\}$ as above. Then, by a form of Shapiro's lemma and flatness of $L_\mathfrak{P}$ over $\mathcal{O}_{L_\mathfrak{P}}$, we get a commutative diagram
\begin{align}\label{eq: comm diagram}\begin{tikzcd}[ampersand replacement=\&]
	{H^1(\mathbf{Q},V_\Pi^*(-a-1)\otimes_{L_\mathfrak{P}} H^{1}(\psi,\mathfrak{n},L_\mathfrak{P}))} \& {H^1(K(\mathfrak{n}),V_\Pi^*(-a-1)(\psi_\mathfrak{P}^{-1}))} \\
	{H^1(\mathbf{Q}, T_\Pi^*(-a-1)\otimes_{\mathcal{O}_{L_\mathfrak{P}}} H^{1}(\psi,\mathfrak{n},\mathcal{O}_{L_\mathfrak{P}}))} \& {H^1(K(\mathfrak{n}),T_\Pi^*(-a-1)(\psi_\mathfrak{P}^{-1}))}
	\arrow["\simeq", from=1-1, to=1-2]
	\arrow[from=2-1, to=1-1]
	\arrow["\simeq", from=2-1, to=2-2]
	\arrow[from=2-2, to=1-2]
\end{tikzcd}\end{align}
\begin{defn}
 For any $\mathfrak{n}\in\mathcal{A}_{cp}$, let $$_c\mathbf{z}_\mathfrak{n}^{\Pi,\psi}\in H^1(K(\mathfrak{n}),V_\Pi^*(-a-1)(\psi_\mathfrak{P}^{-1}))$$ be the image of the class $_cz_\mathfrak{n}^{\Pi,\psi}$ of \emph{\Cref{def: classe over H_n^(p)}}, under the top isomorphism in \eqref{eq: comm diagram}.
 
 If instead $\mathfrak{n}$ is an integral ideal of $K$ coprime to $\mathfrak{f}$ such that $\mathfrak{nf}\in\mathcal{A}_{cp}$, we let $_c\mathbf{z}_\mathfrak{n}^{\Pi,\psi}$ be the image of $_c\mathbf{z}_\mathfrak{nf}^{\Pi,\psi}$ under the Galois corestriction map.
\end{defn}
\begin{thm}\label{thm split norm-relations in Galois}
    Let $c>0$ be an integer coprime to $6N_\Pi N_\mathfrak{f}$. Let $\mathscr{A}$ be the set of integral ideals of $K$ whose norm is coprime to $cpN_\Pi N_\mathfrak{f}$. Then for any $\mathfrak{n},\mathfrak{l}\in\mathscr{A}$ with $\ell=\mathfrak{l\overline{\mathfrak{l}}}\nmid\mathfrak{n}$ a split prime, we have 
    $$\mathrm{cores}^{K(\mathfrak{nl})}_{K(\mathfrak{n})}\left( _c \mathbf{z}_\mathfrak{nl}^{\Pi,\psi} \right)=Q_\mathfrak{l}(\sigma_\mathfrak{l}^{-1})\cdot\ _c \mathbf{z}_\mathfrak{n}^{\Pi,\psi}$$
    where $Q_\mathfrak{l}(X)=\det(1-X\ \mathrm{Frob}_\mathfrak{l}^{-1}|V_\Pi(\psi_\mathfrak{P})(2+a))$, $\mathrm{Frob}_\mathfrak{l}\in\mathrm{Gal}(\overline{K}/K)$ is arithmetic Frobenius and $\sigma_\mathfrak{l}$ is its image in $\mathrm{Gal}(K(\mathfrak{n})/K).$ Moreoever, every such class lies in the image of the cohomology of $T_\Pi^*(-a-1)(\psi_\mathfrak{P}^{-1})$.
    \begin{proof}
        Recall that we are regarding $\psi_\mathfrak{P}$ as a character of $\mathrm{Gal}(\overline{K}/K)$ via the global Artin map normalized in the geometric way. Keeping this in mind, the result follows from \Cref{cor norm-relations one but final}, in the same way as \cite[Theorem $5.3.2$]{Lei_Loeffler_Zerbes_2015}; i.e. the compatibility of the isomorphisms $\nu$ with the norm maps $\mathcal{N}$ given in \eqref{eq: item 4}, and \eqref{eq: comm diagram}.
    \end{proof}
\end{thm}
\begin{defn}
    We write 
    $$_c\mathbf{z}^{\Pi,\psi}\in H^1(K,V_\Pi^*(-a-1)(\psi_\mathfrak{P}^{-1}))$$ for the image of the class $_c\mathbf{z}_{1}^{\Pi,\psi}$ \emph{(}i.e. $\mathfrak{n}=1$\emph{)} under evaluation at the trivial character of $\mathrm{Gal}(K(1)/K)$ where recall that $K(1)$ denotes the maximal abelian $p$-extension inside the Hilbert class field of $K$.
\end{defn}
\section{Bounding Selmer groups}\label{sec: selmer bound}
From now on, we also assume that the prime $p$ fixed at the beginning of \Cref{sec: mapping to Galois} is odd. Let $T$ be an $\mathcal{O}_{L_\mathfrak{P}}$-linear $\mathrm{Gal}(\overline{K}/K)$-representation  that is unramified at almost all primes of $K$ and let $V=T[1/p]$. Let $v$ be a place of $K$ that is coprime to $p$. Then the Bloch-Kato Selmer subspace $H^1_f(K_v,V)$ is defined as in \cite{bloch1990functions}, by
$$H^1_f(K_v,V):=H^1\left(\mathrm{Gal}\left((K_v)^\mathrm{nr}/K_v\right),V^{I_v}\right)$$
   where $I_v$ is the inertia subgroup of $\mathrm{Gal}(\overline{K_v}/K_v)$. We then define $H^1_f(K_v,T)$ as the pre-image of $K^1_f(K_v,V).$ Let $\Sigma_p$ be the set of places of $K$ above $p$. As in \cite[Definition $1.5.1$]{rubin2000annals}, we set
$$\mathrm{Sel}^{\Sigma_p}(K,T):=\mathrm{ker}\left( H^1(K,T)\rightarrow\bigoplus_{v\not\in \Sigma_p} \frac{H^1(K_v,T)}{H^1_f(K_v,T)}\right)$$
and
$$\mathrm{Sel}_{\Sigma_p}(K,T):=\mathrm{ker}\left(\mathrm{Sel}^{\Sigma_p}(K,T)\rightarrow\bigoplus_{v\in\Sigma_p} H^1(K_v,T)\right).
$$
From now on, to ease the notation, we write 
$$T:=T_\Pi^*(-a-1)(\psi_\mathfrak{P}^{-1})\ \ \&\ \ V:=V_\Pi^*(-a-1)(\psi_\mathfrak{P}^{-1})$$
for the representations from previous sections.
\subsection{Hypotheses}\label{sec hypotheses} We now give a list of hypotheses under which the Euler system machinery of \cite{rubin2000annals} applied to our Euler system norm-relations of \Cref{thm split norm-relations in Galois}, gives strict Selmer bounds.
\begin{enumerate}
\item[(H1)] $\Pi$ has level one, and $k_1>k_2.$
\item[(H2)] $T/\mathfrak{P}T$ is irreducible as a $(\mathcal{O}_{L_\mathfrak{P}}/\mathfrak{P})$-linear representation of $\mathrm{Gal}(\overline{K}/K(1))$.
\item[(H3)] If $F$ denotes the field obtained by adjoining $\mu_{p^\infty}$ to $K(1)$, then there exist elements $\tau,\gamma\in\mathrm{Gal}(\overline{K}/F)$ such that $T/(\tau-1)T$ is free of rank one over $\mathcal{O}_{L_\mathfrak{P}},$ and $T^{\gamma=1}=0.$
\end{enumerate}
\begin{rem}
   Regarding hypothesis (H1); assuming $\Pi$ has level one, the condition on the weights, $k_1>k_2,$ implies that certain $H^1_f$'s are as big as possible, locally away from $p$. In general, this happens quite often. Even for $\Pi$ of an arbitrary level, one should be able to obtain analogous sufficient conditions on the weights. This would require us to use a description (of the real part of the poles) of local $L$-factors, and deep local-global compatibility results due to Arthur, for ramified $\Pi_\ell.$ Of course, more work would be required to make this precise. To keep the argument as simple as possible we stick with the level one case, and we will revisit this in the future when we $p$-adically modify and use our Euler system to bound Bloch-Kato Selmer groups.
\end{rem}
\begin{rem}
    Regarding hypotheses (H2) \& (H3); these are standard ``big image'' assumptions for the representation $V_{\Pi}$, and are expected to be satisfied for almost all primes $p$ if $\Pi$ is “sufficiently general”; see for example \cite{weiss2022images}.
\end{rem}
\subsection{The Selmer bound} We set 
$$W:=T\otimes_{\mathcal{O}_{L_\mathfrak{P}}}(L_\mathfrak{P}/\mathcal{O}_{L_\mathfrak{P}})\ \ \& \ \ W^*:=T^\vee(1)$$
where $T^\vee:=\mathrm{Hom}_{\mathcal{O}_{L_\mathfrak{P}}}(T,L_\mathfrak{P}/\mathcal{O}_{L_\mathfrak{P}})$ is the Pontryagian dual of $T$. We write $\ell_{\mathcal{O}_{L_\mathfrak{P}}}(M)\leq \infty$ for the length of an $\mathcal{O}_{L_\mathfrak{P}}$-module $M$. We write $\mathrm{ind}_{\mathcal{O}_{L_\mathfrak{P}}}(_c \mathbf{z}^{\Pi,\psi}), \mathfrak{n}_W$ and $\mathfrak{n}_W^*$ for the quantities defined in \cite[\S $2$]{rubin2000annals}.
\begin{thm}\label{thm selmer bound}
    Suppose that the class $_c \mathbf{z}^{\Pi,\psi}$ is non-zero and the hypotheses of \emph{\Cref{sec hypotheses}} are satisfied. Then 
    $$\ell_{\mathcal{O}_{L_\mathfrak{P}}}\left(\mathrm{Sel}_{\Sigma_p}(K,T^\vee(1))\right)\leq \mathrm{ind}_{\mathcal{O}_{L_\mathfrak{P}}}(_c \mathbf{z}^{\Pi,\psi})+\mathfrak{n}_W+\mathfrak{n}_W^*<\infty.$$
    In particular, the strict Selmer group $\mathrm{Sel}_{\Sigma_p}(K,T^\vee(1))$ is finite.
    \begin{proof}
        Hypothesis (H2) forces $H^1(K,T)$ to be torsion-free, thus \Cref{thm split norm-relations in Galois} holds integrally. So we have obtained Euler system norm-relations for $T$, from $K(\mathfrak{nl})$ to $K(\mathfrak{n})$ for all $\mathfrak{n},\mathfrak{l}\in\mathscr{A}$ with $\ell=\mathfrak{l\overline{\mathfrak{l}}}\nmid\mathfrak{n}$ a split prime. In our case of $K/\mathbf{Q}$ imaginary quadratic, this gives us a large enough supply of primes to carry out Rubin's Euler system argument. However, we need to use the modified version of \cite[Theorem $2.2$]{rubin2000annals} since the composite of all the $K(\mathfrak{n})$ does not contain a $\mathbf{Z}_p$-extension. More precisely, we need to verify that both Hyp$(K,T)$ of \cite[\S $2$]{rubin2000annals}, and condition (ii')(b) of \cite[$\S 9.1$]{rubin2000annals} are satisfied. The former is satisfied automatically by our hypotheses (H2) and (H3). The supply of an element $\gamma$ as in \cite[$\S 9.1$]{rubin2000annals} is included in (H3), and the assumption $T^{\mathrm{Gal}(\overline{K}/K(1))}=0$ follows by (H2) and Nakayama's lemma. The final thing we need to check to fully verify condition (ii')(b) of \cite[$\S 9.1$]{rubin2000annals}, is that our classes $_c\mathbf{z}_\mathfrak{n}^{\Pi,\psi}$ actually lie in $\mathrm{Sel}^{\Sigma_p}(K(\mathfrak{n}),T)$ for every $\mathfrak{n}\in\mathscr{A}$. We now give the details on how to deduce this from hypothesis (H1). 
        
        Let $\mathfrak{n}\in\mathscr{A}$. For starters, it is enough to show that $H^1(K(\mathfrak{n})_{v_\mathfrak{n}},T)=H^1_f(K(\mathfrak{n})_{v_\mathfrak{n}},T)$ for any prime $v_\mathfrak{n}\nmid p$ of $K(\mathfrak{n}).$ Since $H^1_f$ contains torsion, it is enough to show that $H^1(K(\mathfrak{n})_{v_\mathfrak{n}},V)$ vanishes. Using local Tate duality and the fact that $v_\mathfrak{n}\nmid p$, it is enough to show that both $H^1_f(K(\mathfrak{n})_{v_\mathfrak{n}},V)$ and $H^1_f(K(\mathfrak{n})_{v_\mathfrak{n}},V^*(1))$ vanish. These two are isomorphic to $V^{I_{v_\mathfrak{n}}}/(\mathrm{Frob}_{v_\mathfrak{n}}-1)V^{I_{v_\mathfrak{n}}}$ and $V^*(1)^{I_{v_\mathfrak{n}}}/(\mathrm{Frob}_{v_\mathfrak{n}}-1)V^*(1)^{I_{v_\mathfrak{n}}}$ respectively, where $\mathrm{Frob}$ will always denote arithmetic Frobenius, and $I$ the inertia group. Since $\Pi$ has level one, we have that 
        \begin{align}\label{eq: spaces}V^{I_{v_\mathfrak{n}}}=V_\Pi^*(-a-1)\otimes(\psi_\mathfrak{P}^{-1})^{I_{v_\mathfrak{n}}}\ \ \& \ \ V^*(1)^{I_{v_\mathfrak{n}}}=V_\Pi(a+2)\otimes (\psi_\mathfrak{P})^{I_{v_\mathfrak{n}}}.\end{align}
        We write $v$ be the prime of $K$ below $v_\mathfrak{n}$ and $\ell$ for the rational prime below $v$. Recall that $a=k_2-3.$
        
        Firstly, suppose that $v|\mathfrak{f}$; i.e. $\psi_\mathfrak{P}$ is ramified at $v$. By assumption $\mathfrak{n}\in\mathscr{A}$ which means its norm is coprime to $cpN_\Pi N_\mathfrak{f}$. In particular, $v$ is unramified in $K(\mathfrak{n})$, and hence $(\psi_\mathfrak{P})^{I_{v_\mathfrak{n}}}=(\psi_\mathfrak{P})^{I_{v}}=0$. Thus, both spaces in \eqref{eq: spaces} are zero and we are done.

        Now, we suppose that $v\nmid\mathfrak{f}$. In this case $(\psi_\mathfrak{P})^{I_{v_\mathfrak{n}}}=\psi_\mathfrak{P}$, and both spaces in \eqref{eq: spaces} are simply $V$ and $V^*(1)$ respectively. Thus, it now suffices to show that $1$ is not an eigenvalue of $\mathrm{Frob}_{v_\mathfrak{n}}$ acting on $V$ or $V^*(1).$ To do this, it is enough to study the eigenvalues of $\mathrm{Frob}_v$ on these two representations. We \textit{fix} an isomorphism $\mathbf{C}\overset{\iota}{\simeq}\overline{\mathbf{Q}}_p$ extending the embedding $L\hookrightarrow \overline{\mathbf{Q}}_p$ in \eqref{eq: diagram of fields} which was used to define $\psi_\mathfrak{P}$. Using \Cref{thm galois rep}$(4)$ and the fact that $|\psi(v)|=N_{K/\mathbf{Q}}(v)^{1/2}$, we can calculate that the absolute value under $\iota$ of the eigenvalues of $\mathrm{Frob}_v$ acting on $V$ and $V^*(1)$ is given by $N_{K/\mathbf{Q}}(v)^{(k_1-k_2+2)/2}$ and $N_{K/\mathbf{Q}}(v)^{(k_2-k_1)/2}$ respectively. But since $k_1>k_2$ by assumption, these are never equal to one. This is enough to conclude.
        
    \end{proof}
\end{thm}
\bibliography{citation} 
\bibliographystyle{alpha}

\noindent\textit{Mathematics Institute, Zeeman Building, University of Warwick, Coventry CV4 7AL,
England}.\\
\textit{Email address}: Alexandros.Groutides@warwick.ac.uk
\end{document}